\documentclass[a4paper,12pt]{amsart}

\usepackage{amsmath,graphics}
\usepackage{amssymb}
\usepackage{amsfonts}
\usepackage{latexsym}
\usepackage{eucal}
\usepackage[dvips]{graphicx}
\usepackage{color}
\usepackage{enumerate}
\usepackage{mathrsfs}
\newtheorem{thm}{Theorem}[section]
\newtheorem{defi}[thm]{Definition}
\newtheorem{propo}[thm]{Proposition}
\newtheorem{lem}[thm]{Lemma}
\newtheorem{cor}[thm]{Corollary}

\renewcommand{\Re}{{\rm Re}}
\renewcommand{\Im}{{\rm Im}}
\newcommand{\R}{\mathbb{R}}
\newcommand{\C}{\mathbb{C}}
\newcommand{\Z}{\mathbb{Z}}
\newcommand{\N}{\mathbb{N}}

\newcommand{\D}{ \mathcal{D}}

\newcommand{\hh}{\mathbb{H}^{2}}

\newcommand{\G}{{\bf G}}

\newcommand{\lt}{{\mathcal L}}

\newcommand{\Fp}{{\mathbb F}_p}
\newcommand{\half}{{\textstyle{\frac{1}{2}}}}
\newcommand{\rr}{{ \varrho}}
\newcommand{\QEDB}{\hfill\ensuremath{\square}}

\begin{document}
\bibliographystyle{plain}
\title[Large Covers and hyperbolic surfaces]{Large covers and sharp resonances of hyperbolic surfaces}
\keywords{Convex co-compact fuchsian groups, Hyperbolic surfaces, Laplacian, Resonances, Selberg zeta function, L-functions, Representation theory}

\author[Dmitry Jakobson]{Dmitry Jakobson}
\address{McGill University\\
Department of Mathematics and Statistics\\
805 Sherbrooke Street West\\
Montreal, Quebec, Canada H3A0B9 
}
\email{jakobson@math.mcgill.ca}
\author{Fr\'ed\'eric Naud}
\address{Laboratoire de Math\'ematiques d'Avignon \\
Campus Jean-Henri Fabre, 301 rue Baruch de Spinoza\\
84916 Avignon Cedex 9, France. }
\email{frederic.naud@univ-avignon.fr}
\author{Louis Soares}
\address{  Friedrich-Schiller-Universit\"at Jena \\
Institut f\"ur Mathematik \\
Ernst-Abbe-Platz 2, 07743 Jena Germany . 
}
\email{louis.soares@uni-jena.de}

 \maketitle
\begin{abstract} Let $\Gamma$ be a convex co-compact discrete group of isometries of the hyperbolic plane $\hh$, and $X=\Gamma \backslash \hh$
the associated surface. In this paper we investigate the behaviour of resonances of the Laplacian $\Delta_{\widetilde{X}}$ for large degree covers of $X$ given by $\widetilde{X}=\widetilde{\Gamma} \backslash \hh$ where $\widetilde{\Gamma}\vartriangleleft \Gamma$ is a finite index normal subgroup of $\Gamma$. Using various techniques of thermodynamical formalism and representation theory, we prove two new existence results of "sharp non-trivial resonances" close to $\{\Re(s)=\delta \}$, both in the large degree limit, for abelian covers and infinite index congruence subgroups of $SL_2(\Z)$.
 \end{abstract}
 \tableofcontents
 \newpage
 \section{Introduction and results}
 In mathematical physics, resonances generalize the $L^2$-eigenvalues in situations where the underlying geometry is non-compact. Indeed, when the geometry has infinite volume, the $L^2$-spectrum of the Laplacian is mostly continuous and the natural replacement data for the missing eigenvalues are provided by resonances which arise from a meromorphic continuation of the resolvent of the Laplacian. 
 
To be more specific, in this paper we will work with the positive Laplacian $\Delta_X$ on hyperbolic surfaces $X=\Gamma \backslash \hh$, where $\Gamma$
 is a geometrically finite, discrete subgroup of $PSL_2(\R)$. A good reference on the subject is the book of Borthwick \cite{Borthwick}. Here $\hh$ is the hyperbolic plane endowed with its metric of constant curvature $-1$.
Let $\Gamma$ be a geometrically finite Fuchsian group of isometries acting on $\hh$. This means
that $\Gamma$ admits a finite sided polygonal fundamental domain in $\hh$. We will require that $\Gamma$ has no {\it elliptic} elements different from the identity and that the quotient $\Gamma \backslash \hh$ is of {\it infinite hyperbolic area}. If  $\Gamma$ has no parabolic elements (no cusps), then the group is called convex co-compact. 
We will be working with {\it non-elementary} groups $\Gamma$ so that $X$ is never
a hyperbolic cylinder, a "trivial" case for which resonances can actually be computed.
Under these assumptions, the quotient space 
$X=\Gamma \backslash \hh$ is a Riemann surface (called convex co-compact) whose {\it ends geometry} is well known.
The surface $X$ can be decomposed into a compact surface $N$ with geodesic boundary, called the Nielsen region, on which  ends are glued : funnels and cusps.
We refer the reader to the first chapters of Borthwick \cite{Borthwick} for a description of the metric in the ends. 
The limit set $\Lambda(\Gamma)$ is defined as 
$$\Lambda(\Gamma):=\overline{\Gamma.z}\cap \partial \hh,$$
where $z\in \hh$ is a given point and $\Gamma.z$ is the orbit under the action of $\Gamma$ which accumulates
on the boundary $\partial \hh$. The limit set $\Lambda$ does not depend on the choice of $z$ and its Hausdorff dimension $\delta(\Gamma)$
is the critical exponent of Poincar\'e series \cite{Patterson}. 

\bigskip
The spectrum of $\Delta_X$ on $L^2(X)$ has been described completely by 
Lax and Phillips and Patterson in \cite{LP1,Patterson} as follows: 
\begin{itemize}
\item The half line $[1/4, +\infty)$ is the continuous spectrum.
\item There are no embedded eigenvalues inside $[1/4,+\infty)$.
\item The pure point spectrum is empty if $\delta\leq \half$, and finite and starting at $\delta(1-\delta)$ if $\delta>\half$.
\end{itemize}

Using the above notations, the resolvent 
$$R(s):=(\Delta_X-s(1-s) )^{-1}:L^2(X)\rightarrow L^2(X)$$
is a holomorphic family for $\Re(s) >\half$, except at a finite number of possible poles related to the eigenvalues. From the work of Mazzeo-Melrose 
and Guillop\'e-Zworski \cite{MM, GuiZwor1,GuiZwor2}, it can be meromorphically continued (to all $\C$)  from $C_0^\infty(X)\rightarrow C^\infty(X)$, and poles are called {\it resonances}. We denote
in the sequel by $\mathcal{R}_X$ the set of resonances, written with multiplicities.

\bigskip
To each resonance $s\in \C$ (depending on multiplicity) are associated generalized eigenfunctions (so-called purely outgoing states) $\psi_s \in C^\infty(X)$
which provide stationary solutions of the automorphic {\it wave equation} given by 
$$\phi(t,x)=e^{(s-\half)t}\psi_s(x),$$
$$\left (D_t^2+\Delta_X-\frac{1}{4} \right)\phi=0.$$
From a physical point of view, $\Re(s)-\half$ is therefore a rate of decay while $\Im(s)$ is a frequency of oscillation. Resonances that live the longest are called {\it sharp resonances} and are those for which $\Re(s)$ is the closest to the unitary axis $\Re(s)=\half$. In general, $s=\delta$ is the only explicitly known resonance (or eigenvalue if $\delta>\half$). There are very few effective results on the existence of {\it non-trivial } sharp resonances,
and to our knowledge the best statement so far is due to the authors \cite{JN2}, where it is proved that for all $\epsilon>0$, there are infinitely many resonances in the strip
$$\left \{ \Re(s) >\frac{\delta(1-2\delta)}{2}-\epsilon  \right \}.$$
It is conjectured in the same paper \cite{JN2} that for all $\epsilon>0$, there are infinitely many resonances in the strip $\{ \Re(s)>\delta/2-\epsilon\}$. However, the above result, while proving existence of non-trivial resonances, is typically a high frequency statement and does not provide estimates on the imaginary parts (the frequencies), and it is a notoriously hard problem to locate precisely non-trivial resonances.
The goal of the present work is to obtain a different type of existence result by looking at families of covers of given surface, with large degree. Let us be more specific. Given a {\it finite index normal subgroup } $\widetilde{\Gamma}\vartriangleleft \Gamma$, we denote by
$$\G:=\Gamma /  \widetilde{\Gamma}$$
the (finite) Galois group (or covering group) of the cover $\pi_\G$
$$ \pi_\G:\widetilde{X}=\widetilde{\Gamma} \backslash \hh \rightarrow X=\Gamma \backslash \hh.$$
We have an associated natural projection $r_\G:\Gamma \rightarrow \G$ such that $\mathrm{Ker}(r_\G)=\widetilde{\Gamma}$.
We will denote by $\vert \G \vert$ the cardinality of $\G$, and our purpose is to investigate the presence of non-trivial resonances, 
as $\vert \G \vert$ becomes large. We mention that since $\G$ is a finite group, we have $\Lambda(\Gamma)=\Lambda(\widetilde{\Gamma})$, hence the 
leading resonance $\delta$ remains the same for all finite covers. The end-game of this paper is to produce new resonances close to $\delta$ as $\vert \G \vert$
becomes large and see how the algebraic nature of $\G$ affects their location.

A way to attack any problem on resonances of hyperbolic surfaces is through the Selberg zeta function defined for $\Re(s)>\delta$ by
$$Z_\Gamma(s):=\prod_{\mathcal{C}\in \mathcal{P}} \prod_{k\in \N} \left (1-e^{-(s+k)l(\mathcal{C})} \right), $$
where $\mathcal{P}$ is the set of primitive closed geodesics on $\Gamma \backslash \hh$ and $l(\mathcal{C})$ is the length. This zeta function extends analytically to $\C$ and it is known from the work of Patterson-Perry \cite{PatPerry} that non-trivial zeros of $Z_\Gamma(s)$ are resonances with multiplicities. This zeta function method will be our main tool in the analysis of resonances.

Let $\{ \rr \}$ denote the set
of irreducible complex unitary representations of $\G$,
and given $\rr$ we denote by $\chi_{\rr}=\mathrm{Tr}(\rr)$ its character, $V_\rr$ its linear representation space and we set 
$$d_\rr:=\mathrm{dim}_\C(V_\rr).$$ Our first result is the following.
\begin{thm}
\label{main1} 
Assume that $\Gamma$ is convex co-compact. 
For $\Re(s)>\delta$, consider the L-function defined by
$$L_\Gamma(s,\rr):=\prod_{\mathcal{C}\in \mathcal{P}} \prod_{k\in \N} \det \left (Id_{V_\rr}-\rr(\mathcal{C^k})e^{-(s+k)l(\mathcal{C})} \right),$$
where $\rr(\mathcal{C})$ is understood as $\rr(r_\G(\gamma_{\mathcal{C}}))$ where $\gamma_{\mathcal{C}}\in \Gamma$ is any representative
of the conjugacy class defined by $\mathcal{C}$. Then we have the following facts.
\begin{enumerate}
 \item For all $\rr$ irreducible, $L_\Gamma(s,\rr)$ extends as an analytic function to $\C$.
 \item There exist $C_1,C_2>0$ such that for all $p$ large, all $\rr$ irreducible representation of $\G$, and all $s\in \C$,
 we have
 $$\vert L_\Gamma(s,\rr) \vert \leq C_1 \exp\left(C_2 d_\rr \log(1+d_\rr)(1+ \vert s\vert^2) \right).$$
 \item We have the formula valid for all $s \in \C$,
 $$Z_{\widetilde{\Gamma}}(s)=\prod_{\rr\  \mathrm{irreducible}} \left(L_\Gamma(s,\rr) \right)^{d_\rr}.$$
\end{enumerate}
\end{thm}
Notice that the $L$-function for the {\it trivial representation} is just $Z_\Gamma(s)$ and thus $Z_\Gamma(s)$ is always a factor of
$Z_{\widetilde{\Gamma}}(s)$.  There is a long story of L-functions associated with compact extensions of geodesic flows in negative curvature, see for example \cite{Sarnak,KatSun} and \cite{ParryPollicott1}. In the case of pairs of hyperbolic
pants with symmetries, a similar type of factorization has been considered for numerical purposes by Borthwick and Weich \cite{BW1}.
The above factorization is very similar to the factorization of Dedekind zeta functions as a product of Artin L-functions in the case of number fields. In the context of hyperbolic surfaces with infinite volume, although not surprising, the above statement is new and interesting in itself for various applications. We now describe our two main results which deal with two opposite cases, the first one when the Galois group $\G$ is abelian,  the other when $\G=SL_2({\mathbb F}_p)$, which is as far from abelian as possible.

\subsection{Abelian covers}\label{abelian_covers}

An efficient way to manufacture Abelian covers is to use the first homology group with integral coefficients, 
$$H^1(X,\Z)\simeq \Gamma / [\Gamma,\Gamma],$$
where $[\Gamma,\Gamma]$ is the commutator subgroup of $\Gamma$. Since $\Gamma$ is actually a free group \footnote{It's a pure fact of algebraic topology that the fundamental group of a non-compact surface with finite geometry is free,
see for example \cite{Stillwell}.}  on $m$ symbols (see $\S 2$ for the Schottky
representation in the convex co-compact case), then 
$$H^1(X,\Z)\simeq \Z^m.$$
Let us fix a surjective homomorphism $P:\Gamma \rightarrow \Z^m$, given a sequence of positive integers $N:=(N_1,N_2,\ldots,N_m)$ we obtain a surjective  map $\pi_N$ simply given by
$$\pi_N:\left \{ \Z^m\rightarrow \Z / N_1 \Z \times \Z / N_2 \Z \times \ldots \times \Z / N_m \Z \atop
x=(x_1,\ldots,x_m)\mapsto (x_1 \ \mathrm{mod}\ N_1,\ldots, x_m \ \mathrm{mod}\ N_m) \right. $$
One can then check that
$$\Gamma_N:= \mathrm{Ker}(\pi_N\circ P)$$
is indeed a normal subgroup with Galois group
$$\G=\Z / N_1 \Z \times \Z / N_2 \Z \times \ldots \times \Z / N_m \Z.$$ 
We will first prove the following fact.
\begin{thm}
\label{main4}
Assume that $X=\Gamma\backslash \hh$ has at least one cusp, and consider a sequence of Abelian covers with Galois group $\G_j$ as above
with $\vert \G_j \vert\rightarrow +\infty$. Then for all $\epsilon>0$, one can find $j$ such that $X_j=\Gamma_j\backslash \hh$ has at least one non-trivial
resonance $s$ with $\vert s-\delta\vert \leq \epsilon$. 
\end{thm}
In the case of compact hyperbolic surfaces, this is a known result proved in 1974 by Burton Randol \footnote{although there is no interpretation in terms of abelian covers in this early work.}
 \cite{Randol}. Note that in the compact case, it follows also from min-max techniques and the Buser inequality, see for example in the book of Bergeron \cite[Chapter~3]{Bergeron}. 
In the case of abelian covers of the {\it modular} surface, this fact was definitely first observed by Selberg, see in \cite{Selberg}. For more general compact manifolds, we mention the work of R. Brooks \cite{Brooks}
(based on Cheeger constant) which gives sufficient conditions on the fundamental group that guarantees existence of coverings with arbitrarily small spectral gaps.

The outline of the proof is (not surprisingly) as follows: since there is a cusp, we have $\delta>\half$ and resonances close to $\delta$ are actually $L^2$-eigenvalues. One can then use the fact that Cayley graphs of abelian groups are never expanders combined with some $L^2$ techniques  and Fell's continuity of induction to prove the result, following earlier ideas of Gamburd \cite{Gamburd1}. The proof of Theorem \ref{main4} is rather different than the rest of the paper and is found in the last section.

In the convex co-compact case, we can actually prove a much more precise result which goes as follows.
\begin{thm}
 \label{main3}
 Assume that $\Gamma$ is convex co-compact. 
 Let $X_j:=\Gamma_j \backslash \hh$ be a sequence of Abelian covers with Galois group $\G_j$ as above with $\vert \G_j \vert\rightarrow +\infty$ as 
 $j\rightarrow +\infty$. Then, up to a sequence extraction, there exists a small open set $\mathcal{U}$ with $\delta \in \mathcal{U}\subset \C$ such that for all $j$ large we have $\mathcal{R}_{X_j}\cap \mathcal{U}\subset \R$. 
 Morevover, for all test functions
 $\varphi \in C_0^\infty(\mathcal{U})$, we have
 $$ \lim_{j\rightarrow +\infty} \frac{1}{\vert \G_j \vert} \sum_{\lambda \in \mathcal{R}_{X_j}\cap \mathcal{U}} \varphi(\lambda)=\int_{I} \varphi d\mu,$$ where $\mu$ is a finite positive measure which is absolutely continuous  with respect to Lebesgue on an interval $I=[a,\delta]$ for some $a<\delta$. 
\end{thm}
\begin{itemize}

\item The absolutely continuous measure $\mu$ 
depends dramatically on the sequence of covers: a more detailed description of the density is provided in $\S 3$. 
\item Since $\delta$ belongs to the support of $\mu$, a simple approximation argument shows that for all $\varepsilon>0$ small enough, we have as $j\rightarrow +\infty$,
$$\#\{ \lambda \in \mathcal{R}_{X_j}\ :\ \vert \lambda-\delta \vert <\varepsilon\}\sim C_\varepsilon \vert \G_j \vert,$$
for some constant $C_\varepsilon>0$.
\item Another obvious corollary is that for all $\epsilon >0$ one can find a finite Abelian cover $X_j$ of $X$ such that $X_j$ has a non trivial resonance $\epsilon$-close to $\delta$. Both Theorems \ref{main4} and \ref{main3} fully cover the case of all geometrically finite surfaces. We have existence of surfaces with arbitrarily small spectral gap, which was {\it not known so far}. \item Note that the non-trivial resonances obtained here are real: for $\delta>\half$, this is clear because when close enough to $\delta$ they are actually $L^2$-eigenvalues. However when $\delta\leq \half$, this is not an obvious fact.
\item In the general context of scattering theory on spaces with negative curvature, it is to our knowledge the first {\it exact} asymptotic result on the distribution of resonances, apart from the "trivial" cases of elementary groups or cylindrical manifolds where resonances can be explicitly computed.  For a review of the current knowledge on counting results for resonances, we refer to the recent exhaustive survey of Zworski \cite{Zworski}.
\item By using the techniques of \cite{Naud2,OhWinter}, it is likely that one can replace the small neighborhood $\mathcal{U}$ by a thin uniform strip 
$\{ \vert \Re(s)-\delta \vert \leq \varepsilon\}$ for some $\varepsilon>0$. One would need to show uniform (with respect to the cover) "essential spectral gaps" which is
something that has been achieved for congruence covers in \cite{OhWinter}. This should be pursued elsewhere.
\end{itemize}
The proof mostly uses thermodynamical formalism and $L$-functions to analyse carefully the contribution of $L$-factors related to characters which are close to
the identity. In particular we use in a fundamental way dynamical $L$-functions related to characters of $\Z^m$ and their representation as Fredholm
determinants of suitable transfer operators, see $\S 3$. We point out that using the coding available for compact hyperbolic surfaces \cite{Pollicott1}, the proof of the above equidistribution
results carries through without modification in the compact case which is to our knowledge also new (though less surprising). In \cite{Randol}, it was shown that the number of small eigenvalues in $[0,1/4]$ can be as large as wanted by moving to a finite Abelian cover. However his technique based on the "twisted" trace formula prevented him from investigating further the distribution of these small eigenvalues.

\subsection{Congruence subgroups}
Let $\Gamma$ be an infinite index, finitely generated, free subgroup of $SL_2(\Z)$, without parabolic elements. Because $\Gamma$ is free, the projection map
$\mathcal{\pi}:SL_2(\R)\rightarrow PSL_2(\R)$ is injective when restricted to $\Gamma$ and we will thus identify $\Gamma$ with $\pi(\Gamma)$, i.e.
with its realization as a Fuchsian group. Under the above hypotheses, $\Gamma$ is a convex co-compact group of isometries. For all $p>2$ a prime number, we define the congruence subgroup $\Gamma(p)$ by 
$$\Gamma(p):=\{ \gamma \in \Gamma \ :\ \gamma \equiv \mathrm{Id}\ \mathrm{mod}\ p\},$$
and we set $\Gamma(0)=\Gamma$.
Recently, these "infinite index congruence subgroups" have attracted a lot of attention because of the key role they play in number theory and graph theory. We mention the early work of Gamburd \cite{Gamburd1} and the more recent by Bourgain-Gamburd-Sarnak \cite{BGS}, Bourgain-Kontorovich \cite{Kontorovich1} and Oh-Winter 
\cite{OhWinter}. In all of the previously mentioned works, the spectral theory of surfaces
$$X_p:=\Gamma(p)\backslash \hh,$$
plays a critical role and knowledge on resonances is mandatory.
It should be stressed at this point that unlike in the case of Abelian covers treated above, there is a {\it uniform spectral gap} as $p\rightarrow +\infty$,
see \cite{Gamburd1, BGS, OhWinter}, so it's a completely different situation where the non-commutative nature of $\G$ makes it much more difficult to exhibit
new non-trivial resonances in the large $p$ limit.

In \cite{JN1}, the authors have started investigating the behaviour of resonances 
in {\it the large $p$ limit} and the present paper goes in the same direction with different techniques involving sharper tools of representation theory.

Note that it is known from Gamburd \cite{Gamburd1}, that the map 
$$\pi_p:\left \{ \Gamma \rightarrow  SL_2(\Fp) \atop \gamma \mapsto \gamma\ \mathrm{mod}\ p \right.$$
is onto for all $p$ large, and we thus have a family of Galois covers $X_p\rightarrow X$ with Galois group $\G=SL_2(\Fp)$. 
In \cite{JN1}, by combining trace formulae techniques with some a priori upper bounds for $Z_{\Gamma(p)}(s)$ obtained via transfer operator techniques,
we proved the following fact. For all $\epsilon>0$, there exists $C_\epsilon>0$ such that for all $p$ large enough,  
$$C_\epsilon^{-1} p^3\leq \#\mathcal{R}_{X_p}\cap \{ \vert s\vert \leq (\log(p))^\epsilon\}\leq C_\epsilon p^3 (\log(p))^{1+2\epsilon}.$$
We point out that $p^3\asymp \mathrm{Vol}(N_p)$, where $\mathrm{Vol}(N_p)$ is the volume of the convex core of $X_p$, therefore these bounds can be thought as a {\it Weyl law} in the large $p$ regime. 

In the case of covers of compact or finite volume manifolds, after the pioneering work of Heinz Huber \cite{Huber},
precise results for the Laplace spectrum in the "large degree" limit have been obtained in the past
in \cite{Degeorge, Donnelly}. We also mention the recent work \cite{LM} where a precise asymptotic is proved for sequences of compact hyperbolic surfaces. In the case of infinite volume hyperbolic manifolds, we also mention the density bound obtained by Oh \cite{Oh1}.

While this result has near optimal upper and lower bounds, it does not provide a lot of information
on the precise location of non trivial resonances. The second main result of this paper is as follows.

\begin{thm}
\label{main2}
Using the above notations, assume that $\delta>\frac{3}{4}$. Then for all $\epsilon,\beta>0$, and
for all $p$ large,
$$\# \mathcal{R}_{X_p}\cap \left \{ \delta-\textstyle{\frac{3}{4}}-\epsilon\leq \Re(s)\leq \delta\ \mathrm{and}\ 
\vert \Im(s) \vert \leq \left (\log(\log(p))\right)^{1+\beta}\right \} \geq \frac{p-1}{2}.$$
\end{thm}

\bigskip
\begin{itemize}
\item Existence of convex co-compact subgroups $\Gamma$ of $SL_2(\Z)$ with $\delta_\Gamma$ arbitrarily close to $1$ is guaranteed by a theorem
of Lewis Bowen \cite{LewisBowen}. See also \cite{Gamburd1} for some hand-made examples. 
\item The point of Theorem \ref{main2} is that we manage to produce non-trivial resonances without having to affect $\delta$, just by moving to a finite cover, and despite the uniform spectral gap. In that sense, our result is somehow complementary to the spectral gap obtained by Gamburd \cite{Gamburd1}.
\item It would be interesting to know if the $\log\log$ bound can be improved to a constant, but this should require different techniques (see the remarks at end of the main proof).  
\item It is rather clear to us that the methods of proof are robust enough
to allow extensions to more general subgroups of arithmetic groups, in the spirit of the recent work of Magee \cite{Magee}, as long as some knowledge of the group structure of the Galois group $\G$ is available (see $\S 5$.) 
\end{itemize}
\bigskip
The outline of the proof is as follows. Having established the factorization formula, we first notice that since the dimension of any non trivial representation of $\G$ is at least $\textstyle{\frac{p-1}{2}}$, it is enough to show that at least one of the $L$-functions
$L_\Gamma(s,\rr)$ vanishes in the described region as $p\rightarrow \infty$. We achieve this goal through an averaging technique (over irreducible $\rr$)
which takes in account the "explicit" knowledge of the conjugacy classes of $\G$, 
together with the high multiplicities in the length spectrum of $X$. Unlike in finite volume cases where one can 
take advantage of a precise location of the spectrum (for example by assuming GRH), none of this strategy applies here which makes it much harder to mimic existing techniques from analytic number theory.

\bigskip
{\noindent \bf Acknowledgements}. 
Dima Jakobson and Fr\'ed\'eric Naud are supported by ANR grant "GeRaSic". DJ was partially supported by NSERC, FRQNT and Peter Redpath fellowship. FN is supported by Institut Universitaire de France. We all thank Anke Pohl for many helpful discussions.

\section{Vector valued transfer operators and analytic continuation}

\subsection{Bowen-Series coding and transfer operator}
The goal of this section is to prove Theorem \ref{main1}. The technique follows closely previous works \cite{Naud1,JN1} with the notable addition that
we have to deal with vector valued transfer operators. We start by recalling Bowen-Series coding and holomorphic function spaces needed for our analysis. Let $\hh$ denote the Poincar\'e upper half-plane 
$$\hh=\{ x+iy\in \C\ :\ y>0\}$$
endowed with its standard metric of constant curvature $-1$
$$ds^2=\frac{dx^2+dy^2}{y^2}.$$ 
The group of isometries of $\hh$ is $\mathrm{PSL}_2(\R)$ through the action of 
$2\times 2$ matrices viewed as M\"obius transforms
$$z\mapsto \frac{az+b}{cz+d},\ ad-bc=1.$$ 
Below we recall the definition of Fuchsian Schottky groups which will be used to define transfer operators.
A Fuchsian Schottky group is a free subgroup of $\mathrm{PSL}_2(\R)$ built as follows. Let $\D_1,\ldots, \D_m,\D_{m+1},\ldots, \D_{2m}$, $m\geq 2$, 
be $2m$ Euclidean {\it open} discs in $\C$ orthogonal to the line $\R\simeq \partial \hh$. We assume that for all $i\neq j$, $\overline{\D_i} \cap \overline{\D_j}=\emptyset$. 
Let $\gamma_1,\ldots,\gamma_m \in \mathrm{PSL}_2(\R)$ be $m$ isometries such that for all $i=1,\ldots,m$, we have
$$\gamma_i(\D_i)=\widehat{\C}\setminus \overline{\D_{m+i}},$$
where $\widehat{\C}:=\C\cup \{ \infty \}$ stands for the Riemann sphere. For notational purposes, we also set $\gamma_i^{-1}=:\gamma_{m+i}$.
\begin{center}
\includegraphics[scale=0.65]{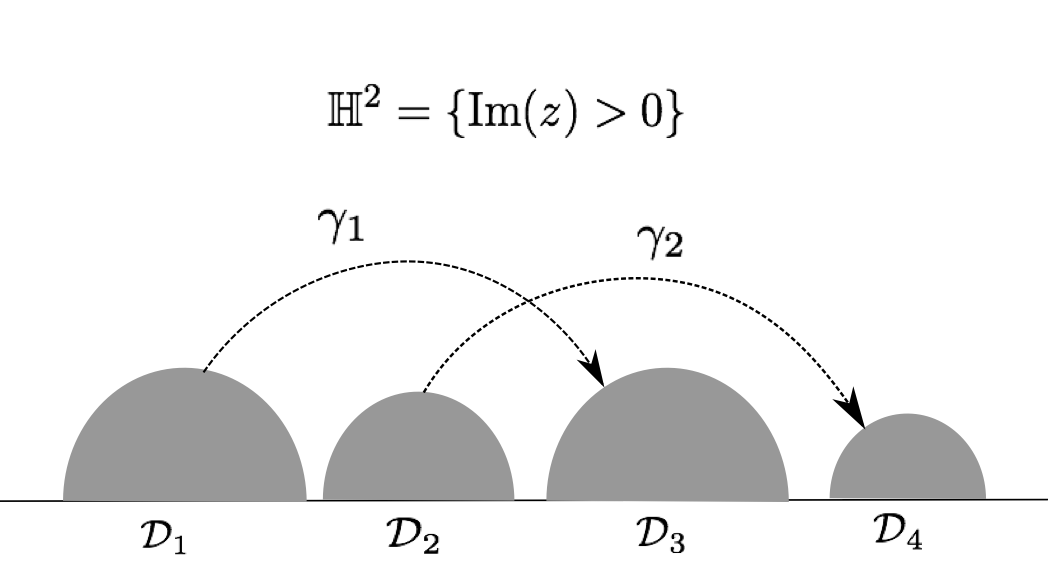}
\end{center}

\bigskip \noindent
Let $\Gamma$ be the free group generated by $\gamma_i,\gamma_i^{-1}$ for $i=1,\ldots,m$, then $\Gamma$ is a convex co-compact group, i.e. it is finitely generated and has no non-trivial parabolic element. The {\it converse is true} : up to isometry, convex co-compact hyperbolic surfaces
can be obtained as a quotient by a group as above, see \cite{Button}. 

For all $j=1,\ldots,2m$, set $I_j:=\D_j\cap \R$. One can define a map 
$$T:I:=\cup_{j=1}^{2m}I_j\rightarrow\R\cup\{\infty\}$$
by setting
$$T(x)=\gamma_j(x)\ \mathrm{if}\ x\in I_j.$$
This map encodes the dynamics of the full group $\Gamma$, and is called the Bowen-Series map, see 
\cite{BowenSeries1, Bowen1} for the genesis of these type of coding. The key properties are orbit equivalence and
uniform expansion of $T$ on the maximal invariant subset $\cap_{n\geq 1} T^{-n}(I)$ which coincides with the limit set
$\Lambda(\Gamma)$, see \cite{Bowen1}.

We now define the function space and the associated transfer operators.
Set 
$$\Omega:=\cup_{j=1}^{2m}\D_j.$$
 Each complex representation space $V_\rr$ is endowed with an inner product $\langle .,.\rangle_\rr$ which makes each representation
$$\rr: \G\rightarrow \mathrm{End}(V_\rr)$$
unitary, where we use the notations of $\S 1$ i.e. $\G$ is the Galois group of the cover $\pi_\G:\widetilde{X}\rightarrow X$,
and we have the associated natural projection $r_\G:\Gamma \rightarrow \G$ such that $\mathrm{Ker}(r_\G)=\widetilde{\Gamma}$.

Consider now the Hilbert space $H^2_\rr(\Omega)$ which is defined as the set of {\it vector valued holomorphic functions} 
$F:\Omega\rightarrow V_\rr$ such that
$$\Vert F\Vert_{H^2_\rr}^2:=\int_{\Omega} \Vert F(z)\Vert^2_\rr dm(z)<+\infty, $$
where $dm$ is Lebesgue measure on $\C$. On the space $H^2_\rr(\Omega)$, we define a "twisted" by $\rr$ transfer operator $\lt_{\rr,s}$ by
$$\lt_{\rr,s}(F)(z):=\sum_{j} ((T')(T_j^{-1}))^{-s}F(y)\rr(T_j^{-1})=
\sum_{j\neq i} (\gamma_j')^s F(\gamma_j z)\rr(\gamma_j),\ \mathrm{if}\ z\in \D_i,$$
where $s\in \C$ is the spectral parameter. Here $\rr(\gamma_j)$ is understood as
$$\rr(r_\G(\gamma_j)),\ \gamma_j\in SL_2(\Z).$$
We also point out that the linear map $\rr(g)$ acts "on the right" on vectors $U \in V_\rr$ simply by fixing an orthonormal basis $\mathcal{B}=(e_1,\ldots,e_{d_\rr})$ of $V_\rr$ and setting
$$U\rr(g):=(U_1,\ldots,U_{d_\rr}) {\mathcal Mat}_{\mathcal{B}}(\rho(g)).$$
Notice that for all $j\neq i$, $\gamma_j:\D_i\rightarrow \D_{m+j}$ is a holomorphic contraction since $\overline{\gamma_j(\D_i)}\subset \D_{m+j}$.
Therefore, $\lt_{\rr,s}$ is a compact {\it trace class} operator and thus has a {\it Fredholm determinant}. We start by recalling a few facts.

We need to introduce some more notations. Considering a finite sequence $\alpha$ with
\[\alpha=(\alpha_1,\ldots,\alpha_n)\in \{1,\ldots, 2m\}^n,\]
we set 
\[ \gamma_\alpha:=\gamma_{\alpha_1}\circ \ldots \circ \gamma_{\alpha_n}. \]
We then denote by $\mathscr{W}_n$ the set of admissible sequences of length $n$ by
\[ \mathscr{W}_n:=\left \{ \alpha \in \{1,\ldots, 2m\}^n\ :\ 
\forall\ i=1,\ldots,n-1,\ \alpha_{i+1}\neq \alpha_i +m\ \mathrm{mod}\ 2m \right \}.\]
The set $\mathscr{W}_n$ is simply the set of reduced words of length $n$.
For all $j=1,\ldots, 2m$, we define $\mathscr{W}_n^j$ by
\[ \mathscr{W}_n^j:=\{ \alpha \in \mathscr{W}_n\ :\ \alpha_n\neq j \}. \] 
If $\alpha \in \mathscr{W}_n^j$, then $\gamma_\alpha$ maps $\overline{\D_j}$ into $\D_{\alpha_1+m}$. Using this set of notations, we have the formula for all
$z\in \D_j$, $j=1,\ldots,2m$, 
$$\lt_{\rr,s}^N(F)(z)=\sum_{\alpha \in \mathscr{W}_N^j} (\gamma_\alpha'(z))^s F(\gamma_\alpha z)\rr(\gamma_\alpha).$$
A key property of the contraction maps $\gamma_\alpha$ is that they are {\it eventually uniformly contracting}, see \cite{Borthwick}, prop 15.4 : there exist $C>0$ and $0<\rho_2<\rho_1<1$ such that for all $z\in \D_j$, for all $\alpha \in \mathscr{W}_n^j$ we have for all $n\geq 1$, 
\begin{equation}
\label{Ucont}
C^{-1}\rho_2^N\leq \sup_{z\in \D_j} \vert \gamma_\alpha'(z) \vert\leq C \rho_1^n
\end{equation}
In addition, they have the {\it bounded distortion property} (see \cite{Naud1} for proofs):
There exists $M_1>0$ such that for all $n,j$ and all $\alpha \in \mathscr{W}_n^j$, we have for all $z \in \D_j$,
\begin{equation}
\label{Bdist}
\left \vert \frac{\gamma''_\alpha(z)}{\gamma'_\alpha(z)}\right \vert \leq M_1.
\end{equation}
We will also need to use the {\it topological pressure} as a way to estimate certain weighted sums over words. 
We will rely on the following fact \cite{Naud1}. Fix $\sigma_0 \in \R$, then there exists $C(\sigma_0)$ such that for all $n$ and 
$\sigma\geq \sigma_0$, we have
\begin{equation}
\label{Pressure0}
\sum_{j=1}^{2m} \left ( \sum_{\alpha \in \mathscr{W}_n^j} \sup_{\D_j} \vert \gamma'_\alpha \vert^\sigma \right) 
\leq C_0e^{nP(\sigma_0)}.
\end{equation}
Here $\sigma\mapsto P(\sigma)$ is the {\bf topological pressure}, which is a strictly convex decreasing function which vanishes at $\sigma=\delta$, see \cite{Bowen1}.  In particular, whenever $\sigma>\delta$, we have $P(\sigma)<0$.
A definition of $P(\sigma)$ is by a variational formula:
$$P(\sigma)=\sup_{\mu} \left ( h_\mu(T)-\sigma \int_{\Lambda}\log \vert T' \vert d\mu \right), $$
where $\mu$ ranges over the set of $T$-invariant probability measures, and $h_\mu(T)$ is the measure theoretic entropy.
For general facts on topological pressure and thermodynamical formalism we refer to \cite{ParryPollicott2}.
Since we will only use it once for the spectral radius estimate below, we don't feel the need to elaborate more on various other definitions of the topological pressure and refer the reader to the above references.
\bigskip
\subsection{Norm estimates and determinant identity}
We start with an a priori norm estimate that will be used later on, see also \cite{JN1} where a similar bound (on a different function space) is proved in appendix.
\begin{propo}
\label{norm1}
Fix $\sigma=\Re(s) \in \R$, then there exists $C_\sigma>0$, independent of $\G,\rr$ such that for all $s \in \C$  with $\Re(s)=\sigma$ and all 
$N$ we have
$$\Vert \lt_{\rr,s}^N\Vert_{H^2_\rr}\leq C_\sigma e^{C_\sigma \vert \Im(s)\vert} e^{2NP(\sigma)}.$$
\end{propo}

\begin{proof}
First we need to be more specific about the complex powers involved here. First we point out that given $z\in \D_i$ then for all $j\neq i$, $\gamma'_j(z)$ belongs to $\C\setminus (-\infty,0]$, simply because each $\gamma_j$ is in $PSL_2(\R)$. 
This make it possible to define $\gamma'_j(z)^s$ by 
$$\gamma'_j(z)^s:=e^{s\mathbb{L}(\gamma'_j(z))},$$
where $\mathbb{L}(z)$ is the complex logarithm defined  on $\C\setminus (-\infty,0]$ by the contour integral
$$\mathbb{L}(z):=\int_1^z\frac{d\zeta}{\zeta}.$$
By analytic continuation, the same identity holds for iterates.
In particular, because of bound (\ref{Ucont}) and also bound (\ref{Bdist}) one can easily show that there exists $C_1>0$ such that for all $N,j$ and all 
$\alpha \in \mathscr{W}_N^j$, we have
\begin{equation}
\label{Bound1}
\sup_{z\in \D_j} \vert \gamma'_\alpha(z)^s\vert \leq e^{C_1\vert \Im(s)\vert}\sup_{\D_j}\vert \gamma'_\alpha \vert^\sigma,
\end{equation}
where $\sigma=\Re(s)$.
We can now compute, given $F\in H^2_\rr(\Omega)$,
$$\Vert \lt_{\rr,s}^N(F)\Vert^2_{H^2_\rr}:=\sum_{j=1}^{2m}\sum_{\alpha,\beta \in \mathscr{W}_N^j}
\int_{\D_j} \gamma'_\alpha(z)^s \overline{\gamma'_\beta(z)^s} 
\langle F(\gamma_\alpha z) \rr(\gamma_\alpha),F(\gamma_\beta z)\rr(\gamma_\beta) \rangle_\rr dm(z).$$
By unitarity of $\rr$ and Schwarz inequality we obtain
$$\Vert \lt_{\rr,s}^N(F)\Vert^2_{H^2_\rr}\leq  e^{2C_1\vert \Im(s)\vert}\sum_j \sum_{\alpha,\beta}
\sup_{\D_j}\vert \gamma'_\alpha \vert^\sigma \sup_{\D_j}\vert \gamma'_\beta \vert^\sigma
\int_{D_j} \Vert F(\gamma_\alpha z) \Vert_\rr \Vert F(\gamma_\beta z) \Vert_\rr dm(z).$$ 
We now remark that $z\mapsto F(z)$ has components in $H^2(\Omega)$, the Bergman space of $L^2$ holomorphic functions
on $\Omega=\cup_j \D_j$, so we can use the scalar reproducing kernel $B_{\Omega}(z,w)$ to write (in a vector valued way)
$$F(\gamma_\alpha z )=\int_{\Omega}F(w) B_{\Omega}(\gamma_\alpha z,w)dm(w).$$
Therefore we get
$$\Vert F(\gamma_\alpha z )\Vert_\rr\leq \int_{\Omega}\Vert F(w)\Vert_\rr \vert B_{\Omega}(\gamma_\alpha z,w)\vert dm(w), $$
and by Schwarz inequality we obtain
$$\sup_{z\in \D_j}  \Vert F(\gamma_\alpha z )\Vert_\rr\leq \Vert F \Vert_{H^2_\rr} 
\left (\int_{\Omega} \vert B_{\Omega}(\gamma_\alpha z,w)\vert^2 dm(w) \right )^{1/2}.$$
Observe now that by uniform contraction of branches $\gamma_\alpha:\D_j\rightarrow \Omega$,  there exists a compact
subset $K\subset \Omega$ such that for all $N,j$ and $\alpha\in \mathscr{W}_{N}^j$,
$$\gamma_\alpha(\D_j)\subset K.$$
We can therefore bound 
$$\int_{\Omega} \vert B_{\Omega}(\gamma_\alpha z,w)\vert^2 dm(w)\leq C$$
uniformly in $z,\alpha$. We have now reached
$$ \Vert \lt_{\rr,s}^N(F)\Vert^2_{H^2_\rr}\leq  \Vert F \Vert_{H^2_\rr}^2 C_2e^{2C_1\vert \Im(s)\vert}\sum_j \sum_{\alpha,\beta}
\sup_{\D_j}\vert \gamma'_\alpha \vert^\sigma \sup_{\D_j}\vert \gamma'_\beta \vert^\sigma,$$
and the proof is now done using the topological pressure estimate (\ref{Pressure0}).
\end{proof}

The main point of the above estimate is to obtain a bound which is independent of $d_\rr$. In particular the spectral radius 
$\rho_{sp}(\lt_{\rr,s})$ of $\lt_{\rr,s}:H^2_\rr(\Omega)\rightarrow H^2_\rr(\Omega)$ is bounded by
\begin{equation}
\label{radius1}
\rho_{sp}(\lt_{\rr,s})\leq e^{P(\Re(s))},
\end{equation}
which is uniform with respect to the representation $\rr$, and also shows that it is a contraction whenever $\sigma=\Re(s)>\delta$.
Notice also that using the variational principle for the topological pressure, it is possible to show that there exist $a_0,b_0>0$
such that for all $\sigma \in \R$,
\begin{equation}
\label{Pressure2}
P(\sigma)\leq a_0-\sigma b_0.
\end{equation}

We continue with a key determinantal identity. We point out that representations of Selberg zeta functions
as Fredholm determinants of transfer operators have a long history going back to Fried \cite{Fried1}, Pollicott 
\cite{Pollicott1} and also Mayer \cite{Mayer1,Mayer2} for the Modular surface. For more recent works involving transfer operators and unitary representations we also mention \cite{Pohl1,Pohl2}.
\begin{propo}
For all $\Re(s)$ large, we have the identity :
\begin{equation}
\label{eq1}
\det(I-\lt_{\rr,s})=L_\Gamma(s,\rr),
\end{equation}
\end{propo}

\begin{proof}
Remark that the above statement implies analytic continuation to $\C$ of each L-function $L_\Gamma(s,\rr)$, since each $s\mapsto \det(I-\lt_{\rr,s})$ is readily an entire function of $s$. 
For all integer $N\geq 1$, let us compute the trace of $\lt_{\rr,s}^N$. 
Our basic reference for the theory of Fredholm determinants on Hilbert spaces is \cite{Simon}. Let $(e_1,\ldots,e_{d_\rr})$ be an orthonormal
basis of $V_\rr$. For each disc $\D_j$ let $(\varphi_\ell^{j})_{\ell \in \N}$ be a Hilbert basis of the Bergmann space $H^2(\D_j)$, that
is the space of square integrable holomorphic functions on $\D_j$. Then the family defined by
$$\Psi_{j,\ell,k}(z):=\left \{ \varphi_\ell^{j}(z)e_k\ \mathrm{if}\ z\in \D_j \atop 0\  \mathrm{otherwise}, \right.$$
is a Hilbert basis of $H^2_\rr(\Omega)$. Writing
$$\langle \lt_{\rr,s}^N(\Psi_{j,\ell,k}) ,\Psi_{j,\ell,k}\rangle_{H^2_\rr(\Omega)}=
\sum_{\alpha \in \mathscr{W}_N^j} \int_{\D_j} (\gamma_\alpha'(z))^s \varphi^j_\ell(\gamma_\alpha z)\overline{\varphi^j_\ell(z)}
\langle e_k \rr(\gamma_\alpha), e_k \rangle_\rr dm(z),$$
we deduce that
$$\mathrm{Tr}( \lt_{\rr,s}^N)=\sum_{j,\ell,k}\langle \lt_{\rr,s}^N(\Psi_{j,\ell,k}) ,\Psi_{j,\ell,k}\rangle_{H^2_\rr(\Omega)}$$
$$=\sum_j \sum_{\alpha \in \mathscr{W}_N^j \atop \alpha_1=m+j} \chi_\rr(\gamma_\alpha) \int_{\D_j} (\gamma_\alpha'(z))^s B_{\D_j}(\gamma_\alpha z,z)dm(z),$$
where $\chi_\rr$ is the character of $\rr$ and $B_{\D_j}(w,z)$ is the {\it Bergmann reproducing kernel} of $H^2(\D_j)$.
There is an explicit formula for the Bergmann kernel of a disc $\D_j=D(c_j,r_j)$ : 
$$B_{\D_\ell}(w,z)=\frac{r_j^2}{\pi \left [ r_j^2-(w-c_j)(\overline{z}-c_j)\right ]^2}.$$
It is now an exercise involving Stoke's and Cauchy formula (for details we refer to Borthwick \cite{Borthwick}, P. 306) to obtain the Lefschetz identity
$$\int_{\D_j} (\gamma_\alpha'(z))^s
 B_{\D_j}(\gamma_\alpha z,z) dm(z)= \frac{(\gamma_\alpha'(x_\alpha))^s}{1-\gamma_\alpha'(x_\alpha)},$$
 where $x_\alpha$ is the unique fixed point of $\gamma_\alpha:\D_j\rightarrow \D_j$. Moreover, 
 $$\gamma_\alpha'(x_\alpha)=e^{-l( \mathcal{C}_\alpha)},$$
 where $\mathcal{C}_\alpha$ is the closed geodesic represented by the conjugacy class of $\gamma_\alpha \in \Gamma$, and
 $l( \mathcal{C}_\alpha)$ is the length. There is a one-to-one correspondence between prime reduced words (up to circular permutations) in
 $$\bigcup_{N\geq 1} \bigcup_{j=1}^{2m} \{\alpha \in \mathscr{W}_N^j\ \mathrm{such\ that}\ \alpha_1=m+j \},$$
and prime conjugacy classes in $\Gamma$ (see Borthwick \cite[page~303]{Borthwick}), therefore each prime conjugacy class in $\Gamma$ and its iterates appear in the above sum, when $N$ ranges from $1$ to $+\infty$.
 
We have therefore reached formally (absolute convergence is valid for $\Re(s)$ large, see later on)
$$\sum_{N\geq 1} \frac{1}{N}\mathrm{Tr}( \lt_{\rr,s}^N)=\sum_{N\geq 1} \frac{1}{N}\sum_j \sum_{\alpha \in \mathscr{W}_N^j \atop \alpha_1=m+j}
\chi_\rr(\gamma_\alpha) \frac{(\gamma_\alpha'(x_\alpha))^s}{1-\gamma_\alpha'(x_\alpha)}$$
$$=\sum_{\mathcal{C}\in \mathcal{P}} \sum_{k\geq 1} \frac{\chi_\rr(\mathcal{C}^k)}{k}
\frac{e^{-skl(\mathcal{C})}}{1-e^{-kl(\mathcal{C})}}.$$
The prime orbit theorem for convex co-compact groups says that as $T\rightarrow +\infty$, (see for example \cite{Lalley,Naud3}), 
$$\#\{ (k,\mathcal{C})\in \N_0\times \mathcal{P}\ :\ kl(\mathcal{C})\leq T\}=\frac{e^{\delta T}}{\delta T}\left (1+o(1)\right).$$
On the other hand, since $\chi_\rr$ takes obviously finitely many values on $\G$
we get absolute convergence of the above series for $\Re(s)>\delta$. For all $\Re(s)$ large, we get again formally
$$\det(I-\lt_{\rr,s})=\mathrm{exp}\left( \sum_{N\geq 1} \frac{1}{N}\mathrm{Tr}( \lt_{\rr,s}^N)   \right)$$
$$=\mathrm{exp}\left(-\sum_{\mathcal{C},k,n} \frac{\chi_\rr(\mathcal{C}^k)}{k}
e^{-(s+n)kl(\mathcal{C})} \right)=\prod_{\mathcal{C}\in \mathcal{P}}\prod_{n\in \N}
\mathrm{exp}\left(-\sum_{k\geq 1} \frac{\chi_\rr(\mathcal{C}^k)}{k} e^{-(s+n)kl(\mathcal{C})} \right)$$
$$=\prod_{\mathcal{C}\in \mathcal{P}} \prod_{k\in \N} \det \left (Id_{V_\rr}-\rr(\mathcal{C^k})e^{-(s+k)l(\mathcal{C})} \right).$$
This formal manipulations are justified for $\Re(s)>\delta$ by using the spectral radius estimate (\ref{radius1}) and the fact that if $A$
is a trace class operator on a Hilbert space $\mathcal{H}$ with $\Vert A\Vert_{\mathcal{H}}<1$ then we have 
$$\det(I-A)=\mathrm{exp}\left( -\sum_{N\geq 1} \frac{1}{N}\mathrm{Tr}( A^N)\right),$$
(this is a direct consequence of Lidskii's theorem, see \cite[Chapter~3]{Simon}).
The proof is finished and we have claim $1)$ of Theorem \ref{main1}. 
\end{proof}

\noindent
Claim $3)$ follows from the formula (valid for $\Re(s)>\delta$)
$$ \det(I-\lt_{\rr,s})=\mathrm{exp}\left(-\sum_{\mathcal{C},k,n} \frac{\chi_\rr(\mathcal{C}^k)}{k}
e^{-(s+n)kl(\mathcal{C})} \right),$$
and the identity for the character of the regular representation (see \cite[Chapter~2]{serre}) 
\begin{equation}
\label{Schur1}
\sum_{\rr\  \mathrm{irreducible}} d_\rr \chi_\rr(g)=\vert \G \vert \mathcal{D}_e(g),
\end{equation}
where $\mathcal{D}_e$ is the dirac mass at the neutral element $e$. Indeed, using (\ref{Schur1}), we get
\begin{equation}
\label{form1}
\prod_{\rr\  \mathrm{irreducible}} \left (\det(I-\lt_{\rr,s})\right)^{d_\rr}=
\mathrm{exp}\left(-\vert \G \vert \sum_{k,n} \sum_{\mathcal{C} \in \mathcal{P}\atop r_\G(\mathcal{C})=e }\frac{1}{k}
e^{-(s+n)kl(\mathcal{C})} \right).
\end{equation}
The end of the proof rests on an algebraic fact related to the splitting of conjugacy classes in $\widetilde{\Gamma}$. For the benefit of the reader, we give the outline.
It is easy to check that any prime conjugacy class $\widetilde{\mathcal{C}}$ in $\widetilde{\Gamma}$ has a representative given by (representative of) a power of a prime conjugacy class (in $\Gamma$), i.e.
$$\widetilde{\mathcal{C}}=\mathcal{C}^\ell,$$
for some $1\leq \ell\leq \vert G \vert$. It is then a fact of group theory that the conjugacy class of  $\mathcal{C}^\ell$ in $\Gamma$ will split in $\widetilde{\Gamma}$ in one-to-one correspondence with the cosets of
$$\Gamma/ \widetilde{\Gamma} C_\Gamma(\mathcal{C}^\ell),$$
where $C_\Gamma(\mathcal{C}^\ell)$ is the centralizer in $\Gamma$ of $\mathcal{C}^\ell$. Because we are in a free group, this centralizer is the elementary group generated
by $\mathcal{C}$, which shows that the number of conjugacy classes in $\widetilde{\Gamma}$ is $\vert G \vert /\ell$. This factor $\ell$ is exactly what's needed to recognize
in (\ref{form1}) the length $\ell l(\mathcal{C})=l(\mathcal{C}^\ell)=l(\widetilde{\mathcal{C}})$.

We refer the reader to \cite{PohlFedosova} for more details, including a
complete proof of the factorization formula $3)$ for geometrically finite groups. We point out that this type of analog of the Artin factorization had already been proved by Venkov-Zograv in \cite{VenkovZograf} for cofinite groups.

\subsection{Singular value estimates}
The proof of claim $2)$ will require more work and will use singular values estimates for vector-valued operators.  We now recall a few facts on singular values of trace class operators.
Our reference for that matter is for example the book \cite{Simon}. If $T:\mathcal{H}\rightarrow \mathcal{H}$ is a compact operator acting on a Hilbert space $\mathcal{H}$, the {\it singular value sequence} is by definition the sequence $\mu_1(T)=\Vert T\Vert\geq \mu_2(T)\geq\ldots \geq \mu_n(T)$ of the eigenvalues of the positive self-adjoint operator $\sqrt{T^*T}$. To estimate singular values in a vector valued setting, we will rely on the following fact.
\begin{lem}
\label{singular}
Assume that $(e_j)_{j\in J}$ is a Hilbert basis of $\mathcal{H}$, indexed by a countable set $J$. Let $T$ be a compact operator on $\mathcal{H}$.
Then for any subset $I\subset J$ with $\# I=n$ we have 
$$\mu_{n+1}(T)\leq \sum_{j\in J\setminus I} \Vert T e_j \Vert_{\mathcal H}.$$
\end{lem}

\begin{proof}
By the min-max principle for bounded self-adjoint operators, we have
$$\mu_{n+1}(T)=\min_{\mathrm{dim}(F)=n} \max_{w \in F^{\perp},\Vert w \Vert=1} \langle \sqrt{T^*T}w,w\rangle.$$
Set $F=\mathrm{Span}\{ e_j,\ j\in I\}$. Given $w=\sum_{j\not \in I} c_j e_j$ with $\sum_j \vert c_j \vert^2=1$, we obtain via Cauchy-Schwarz inequality
$$\vert \langle \sqrt{T^*T}w,w\rangle\vert \leq \Vert \sqrt{T^*T}(w)\Vert=\Vert T(w)\Vert\leq \sum_{j\not \in I} \Vert T(e_j)\Vert,$$
which concludes the proof. 
\end{proof}
 
Our aim is now to prove the following bound.
\begin{propo}
\label{eigen1}
Let $(\lambda_k(\lt_{\rr,s}))_{k\geq 1}$ denote the eigenvalue sequence of the compact operators $\lt_{\rr,s}$. There exists $C>0$ and
$0<\eta$ such that for all $s\in \C$ and all representation $\rr$, we have for all $k$,
$$\vert\lambda_k( \lt_{\rr,s}) \vert \leq Cd_\rr e^{C\vert s\vert} e^{-\frac{\eta}{d_\rr}k}.$$
\end{propo}
\noindent Before we prove this bound, let us show quickly how the combination of the above bound with (\ref{radius1}) 
gives the estimate $2)$ of Theorem
\ref{main1}. By definition of Fredholm determinants, we have
$$\log \vert L_\Gamma(s,\rr) \vert\leq \sum_{k=1}^\infty \log(1+\vert\lambda_k( \lt_{\rr,s}) \vert)$$
$$=\sum_{k=1}^N \log(1+\vert\lambda_k( \lt_{\rr,s}) \vert) +\sum_{k=N+1}^\infty \log(1+\vert\lambda_k( \lt_{\rr,s}) \vert),$$
where $N$ will be adjusted later on. The first term is estimated via (\ref{Pressure2}) as
$$ \sum_{k=1}^N \log(1+\vert\lambda_k( \lt_{\rr,s}) \vert)\leq \widetilde{C}(\vert s\vert +1)N,$$
for some large constant $\widetilde{C}>0$. On the other hand we have by the eigenvalue bound from Proposition \ref{eigen1}
$$\sum_{k=N+1}^\infty \log(1+\vert\lambda_k( \lt_{\rr,s}) \vert )\leq \sum_{k=N+1}^\infty \vert\lambda_k( \lt_{\rr,s}) \vert$$
$$\leq Cd_\rr e^{C\vert s\vert} \sum_{k\geq N+1} e^{-\frac{\eta}{d_\rr}k}=
Cd_\rr e^{C\vert s\vert} \frac{e^{-(N+1)\eta/d_\rr}}{1-e^{-\eta/d_\rr}}$$
$$\leq C' \frac{d_\rr^2}{\eta}e^{C\vert s\vert}  e^{-N\frac{\eta}{d_\rr}}.$$
Choosing $N=B[\vert s\vert d_\rr]+B[d\rr \log (d_\rr+1)]$ for some large $B>0$ leads to
$$\sum_{k=N+1}^\infty \log(1+\vert\lambda_k( \lt_{\rr,s}) \vert )\leq \widetilde{B}$$
for some constant $\widetilde{B}>0$ uniform in $\vert s\vert$ and $d_\rr$. Therefore we get 
$$ \log \vert L_\Gamma(s,\rr) \vert\leq O\left(d_\rr \log(d_\rr+1)(\vert s\vert^2+1)\right),$$
which is the bound claimed in statement $2)$.

\begin{proof}[Proof of Proposition \ref{eigen1}]
We first recall that if $\D_j=D(c_j,r_j)$, an explicit Hilbert basis of the Bergmann space $H^2(\D_j)$ is given by
the functions ( $\ell=0,\ldots,+\infty$, $j=1,\ldots,2m$)
$$\varphi_\ell^{(j)}(z)=\sqrt{\frac{\ell+1}{\pi}}\frac{1}{r_j} \left (\frac{z-c_j}{r_j} \right)^\ell.$$
By the Schottky property, one can find $\eta_0>0$ such for all $z\in \D_j$, for all $i\neq j$ we have $\gamma_i(z)\in \D_{i+m}$ and
$$\frac{\vert \gamma_i(z)-c_{m+i}\vert}{r_{m+i}}\leq e^{-\eta_0},$$
so that we have uniformly in $i,z$,
\begin{equation}
\label{Ucont2}
 \vert \varphi_\ell^{(i+m)}(\gamma_i z)\vert \leq Ce^{-\eta_1 \ell},
\end{equation}
for some $0<\eta_1<\eta_0$. Going back to the basis $\Psi_{j,\ell,k}(z)$ of $H^2_\rr(\Omega)$, we can write
$$\Vert \lt_{\rr,s}( \Psi_{j,\ell,k})\Vert^2_{H^2_\rr}=\sum_{n=1}^{2m} \sum_{i,i'\neq n}
\int_{\D_n}(\gamma_i(z))^s \overline{(\gamma_{i'}(z))^s} 
\langle\Psi_{j,\ell,k}(\gamma_i z)\rr(\gamma_i),\Psi_{j,\ell,k}(\gamma_{i'} z)\rr(\gamma_{i'})  \rangle_\rr dm(z).$$
Using Schwarz inequality and unitarity of the representation $\rr$ for the inner product $\langle .,.\rangle_\rr$, 
we get by (\ref{Ucont2}) and also (\ref{Bound1}),
$$ \Vert \lt_{\rr,s}( \Psi_{j,\ell,k})\Vert^2_{H^2_\rr}\leq \widetilde{C}e^{\widetilde{C}\vert s \vert} e^{-2\eta_1 \ell},$$
for some large constant $\widetilde{C}>0$. We can now use Lemma \ref{singular} to write
$$\mu_{2md_\rho n+1}(\lt_{\rr,s})\leq \sum_{j=1}^{2m} \sum_{\ell=n}^{+\infty} \sum_{k=1}^{d_\rr} \Vert \lt_{\rr,s}(\Psi_{j,\ell,k}) \Vert_{H^2_\rr} $$
$$\leq C d_\rho e^{\widetilde{C}\vert s\vert} e^{-\eta_1 n},$$
for some $C>0$. Given $N\in \N$, we write $N=2md_\rr k+r$ where $0\leq r<2md_\rr$ and $k=[\frac{N}{2md_\rho}]$. We end up with
$$\mu_{N+1}(\lt_{\rho,s})\leq\mu_{2md_\rr k+1}(\lt_{\rr,s})\leq C'd_\rr e^{\widetilde{C}\vert s\vert} e^{-\eta_2 N/d_\rr},$$
for some $\eta_2>0$. To produce a bound on the eigenvalues, we use then a variant of Weyl inequalities (see \cite[Theorem~1.14]{Simon}) to get
$$\vert \lambda_N(\lt_{\rr,s})\vert\leq \prod_{k=1}^N \vert \lambda_k(\lt_{\rr,s})\vert \leq \prod_{k=1}^N \mu_k(\lt_{\rr,s}),$$
which yields
$$\vert \lambda_N(\lt_{\rr,s})\vert \leq C_1 d_\rr e^{C_2 \vert s\vert} e^{-\frac{\eta_2}{Nd_\rr}\sum_{k=1}^N k}.$$
Using the well known identity $\sum_{k=1}^N k=\frac{N(N+1)}{2}$ we finally recover
$$\vert \lambda_N(\lt_{\rr,s})\vert \leq C_1 d_\rr e^{C_2 \vert s\vert} e^{-\frac{\eta N}{d_\rr}},$$
for some $\eta>0$ and the proof is done. 
\end{proof}

\section{Equidistribution of resonances and abelian covers}
In this section we prove Theorem \ref{main3}. We use the notations of $\S 1$. We recall that we consider a family of Abelian covers of a fixed surface
$X=\Gamma \backslash \hh$ given by normal subgroups $\Gamma_j \vartriangleleft \Gamma$ with Galois group
$$\G_j=\Z / N_1^{(j)} \Z \times \Z / N_2^{(j)} \Z \times \ldots \times \Z / N_m^{(j)} \Z.$$ 
Since we assume that $\vert \G_j \vert \rightarrow +\infty$ as $j\rightarrow +\infty$, we can extract a sequence (and reindex) such that
$$\G_j=\Z / N_1^{(j)} \Z \times \ldots \times \Z / N_r^{(j)} \Z \times \Z / N_{r+1} \Z \times \ldots \times \Z / N_m \Z,$$
with $\min\{N_1^{(j)},\ldots,N_r^{(j)}\}\rightarrow +\infty$ as $j\rightarrow +\infty$ and $N_{r+1},\ldots N_m$ are fixed (and could be $1$).
The characters of $\G_j$ are given by 
$$\chi_\alpha(g):=\exp\left(2i\pi \sum_{\ell=1}^m \frac{\alpha_\ell}{N_\ell}g_\ell \right),$$
where $g=(g_1,\ldots,g_m)$ and $\alpha=(\alpha_1,\ldots,\alpha_m)$ with $\alpha_\ell \in \{0,\ldots,N_\ell-1\}$. 
Thanks to Theorem \ref{main1} and since the representations are one-dimensional, we have the factorization formula
$$Z_{\Gamma_j}(s)=\prod_{\alpha}L_\Gamma(s,\chi_\alpha),$$
where $\alpha$ belongs to the above specified set product. The case $\alpha=0$ corresponds to the trivial representation, hence the associated $L$-function is $Z_\Gamma(s)$ which has a simple zero at $s=\delta$. Roughly speaking, we need to split this product into two separate factors:
the one corresponding to "small $\alpha$'s" which will produce a zero close to $s=\delta$ via an implicit function theorem, and the other ones
for which we have to show that they do not vanish in a small neighbourhood of $\delta$. To that effect, we will introduce an auxiliary $L$-function that is related to characters of the homology group $H^1(X,\Z)\simeq \Z^m$.

Consider the Hilbert space $H^2(\Omega)$ which is defined as the set of {\it holomorphic functions} 
$F:\Omega\rightarrow \C$ such that
$$\Vert F\Vert_{H^2}^2:=\int_{\Omega} \vert F(z)\vert^2 dm(z)<+\infty, $$
where $dm$ is Lebesgue measure on $\C$. On the space $H^2(\Omega)$, given $\theta \in \C^m$, 
we define a "twisted" transfer operator $\lt_{s,\theta}$ by
$$\lt_{s,\theta}(F)(z):=
\sum_{\ell\neq k} (\gamma_\ell')^s e^{2i\pi \theta \bullet P(\gamma_\ell)}F(\gamma_\ell z),\ \mathrm{if}\ z\in \D_k,$$
where $s\in \C$ is the spectral parameter, $P:\Gamma \rightarrow \Z^m$ is the projection in the first homology group.
In addition we have denoted by $\theta \bullet a$ the pairing
$$\theta \bullet a:=\sum_{k=1}^m \theta_k a_k.$$
This family (of trace class operators) depends holomorphically on $(s,\theta)$, therefore the Fredholm determinant
$$L_\Gamma(s,\theta):=\det(I-\lt_{s,\theta}) $$
is holomorphic on $\C\times \C^m$. When $\theta \in \R^m/\Z^m$, this is actually the $L$-function associated to the obvious character $\chi_\theta$
of $\Z^m$.
Note that using this auxiliary function, we have now
\begin{equation}
\label{factor1}
Z_{\Gamma_j}(s)=\prod_{k=(k_1,\ldots ,k_m)\in \mathcal{S}_j}L_\Gamma(s,k_1/N_1^{(j)},\ldots,k_r/N_r^{(j)},\ldots ,k_m/N_m),
\end{equation}
where
$$\mathcal{S}_j=\{0,\ldots,N_1^{(j)}-1 \}\times \ldots \times \{0,\ldots,N_r^{(j)} -1\}\times \ldots \times \{0,\ldots,N_m-1\}.$$

\subsection{A non vanishing result for $L_\Gamma(s,\theta)$}
The goal of this subsection is to establish the following fact which is crucial in the analysis of resonances close to $s=\delta$.
\begin{propo}
 \label{nonv}
 Using the above notations, we have for $\theta \in \R^m$,
 $$L_\Gamma(\delta,\theta)=0 \Leftrightarrow \theta \in \Z^m.$$
\end{propo}

\begin{proof}
Obviously if $\theta \in \Z^m$, then $L_\Gamma(s,\theta)=Z_\Gamma(s,\theta)$ and vanishes at $s=\delta$.
The converse will follow from a convexity argument that is similar to what has been used by Parry and Pollicott \cite{ParryPollicott2} to analyze 
dynamical Ruelle zeta functions on the line $\{\Re(s)=1\}$, see chapter 5. First we need to recall the usual "normalizing trick" which
is essential in the latter part of the argument. By the Ruelle-Perron-Frobenius Theorem (see \cite[Theorem~2.2]{ParryPollicott2}), the operator
$$\lt_{\delta,0}:H^2(\Omega)\rightarrow H^2(\Omega)$$
has $1$ as a simple eigenvalue and the associated eigenspace is spanned by a real-analytic function $H$ which satisfies
$H(x)>0$ for all $x \in \Lambda(\Gamma)$. By setting (we work on $\Lambda(\Gamma)$)
$$\mathcal{M}_\delta(F)(x):=\sum_{\ell \neq k}e^{ g_\ell(x)} F(\gamma_\ell x),\ x\in I_k\cap \Lambda(\Gamma),$$
where
$$g_\ell(x)=\delta \log(\gamma'_\ell (x))-\log H(x)+\log H(\gamma_\ell x),$$
we obtain an operator 
$$\mathcal{M}_\delta:C^0(\Lambda(\Gamma))\rightarrow C^0(\Lambda(\Gamma))$$
which satisfies $\mathcal{M}_\delta ({\bf 1})={\bf 1 }$. Assume now that $L_\Gamma(\delta,\theta)=0$ for some $\theta \in \R^m$.
Then $\lt_{\delta,\theta}$ has $1$ as an eigenvalue and pick an associated non trivial eigenfunction $W$, obviously continuous 
on $\Lambda(\Gamma)$. By writing
$$H^{-1}\lt_{\delta,\theta}(H. (H^{-1}W))=H^{-1}W, $$
we deduce that 
\begin{equation}
\label{convex1}
\sum_{\ell \neq k}e^{ g_\ell(x)} e^{2i\pi \theta \bullet P(\gamma_\ell)}\widetilde{W}(\gamma_\ell x)
=\widetilde{W}(x),\ x\in I_k\cap \Lambda(\Gamma),
\end{equation}
where we have set
$$\widetilde{W}(x)=H^{-1}(x) W(x). $$
Choosing $x_0 \in \Lambda(\Gamma)$ (say in $I_k\cap \Lambda(\Gamma))$) such that 
$$\vert\widetilde{W}(x_0)\vert =\sup_{\xi \in \Lambda(\Gamma)} \vert\widetilde{W}(\xi)\vert,$$
we get by the triangle inequality
$$\sup_{\xi \in \Lambda(\Gamma)} \vert\widetilde{W}(\xi)\vert\leq 
\mathcal{M}_\delta(  \vert\widetilde{W}\vert)(x_0)\leq  \sup_{\xi \in \Lambda(\Gamma)} \vert\widetilde{W}(\xi)\vert.$$
The same conclusion holds when iterating $\mathcal{M}_\delta$ so that for all $N\geq 0$, we have
$$\sup_{\xi \in \Lambda(\Gamma)} \vert\widetilde{W}(\xi)\vert=
\mathcal{M}_\delta^N(  \vert\widetilde{W}\vert)(x_0).$$
Because $\mathcal{M}^N_\delta$ are normalized, this forces
$$\sup_{\xi \in \Lambda(\Gamma)} \vert\widetilde{W}(\xi)\vert= \vert \widetilde{W}(\gamma_\alpha x_0)\vert$$
for all words $\alpha \in \mathscr{W}_N^k$. By density in $\Lambda(\Gamma)$ as $N\rightarrow +\infty$ of the set of inverse images
$\{ \gamma_\alpha x_0 \}_{\alpha \in \mathscr{W}_N^k}$, we deduce that $\vert  \widetilde{W}\vert$ is constant on $\Lambda(\Gamma)$.
We further assume that 
$$\vert  \widetilde{W}\vert=1.$$
By strict convexity of the unit euclidean ball in $\C$, we deduce from (\ref{convex1}) that for all $\ell \neq k$, we have
$$e^{2i\pi \theta \bullet P(\gamma_\ell)} \widetilde{W}\circ \gamma_\ell(x)=\widetilde{W}(x),\ x \in x\in I_k\cap \Lambda(\Gamma).$$
Writing 
$$\widetilde{W}(x)=e^{2i\pi V(x)},$$
where $V:\Lambda(\Gamma)\rightarrow \R$ is a continuous lift, we end up with the identity
($ x\in I_k\cap \Lambda(\Gamma)$, $\ell\neq k$)
\begin{equation}
 \label{convex2}
 \theta \bullet P(\gamma_\ell)=V(x)-V(\gamma_\ell x)+M_{x,\ell},
\end{equation}
where $M_{x,\ell}$ is $\Z$-valued. Now for each $k=1,\ldots,m$, let $x_k \in I_{m+k}$ be the unique attracting fixed point of 
$$\gamma_k:I_{m+k}\rightarrow I_{m+k}.$$
We get therefore from (\ref{convex2}) that for all $k=1,\ldots,m$
\begin{equation}
\label{convex3}
 \theta \bullet P(\gamma_k)\in \Z.
\end{equation}
Since $\Gamma$ is a free group on $m$ elements generated by $\gamma_1,\ldots,\gamma_m$, then
$$(P(\gamma_1),\ldots,P(\gamma_m))$$
is a $\Z$-basis of $H^1(X,\Z)\simeq \Z^m$. As a consequence, the $m\times m$ matrix whose rows are given by the vectors
$P(\gamma_1),\ldots, P(\gamma_m)$ has determinant $\pm 1$ and is thus invertible with integer coefficients : this implies
by (\ref{convex3}) that $\theta \in \Z^m$, the proof is done. 
\end{proof}

A direct corollary, which is what we will actually use in the proof of Theorem \ref{main3}, is the following.
\begin{cor}
\label{smallz}
Using the above notations, for all $\epsilon>0$ small enough, one can find a complex neighbourhood $\mathcal{V}$ of $\delta$
such that for all $s\in \mathcal{V}$ and $\theta \in \R^m$, 
$$L_\Gamma(s,\theta)=0 \Rightarrow \mathrm{dist}(\theta, \Z^m)< \epsilon.$$
\end{cor}

\begin{proof}
Argue by contradiction. Fix some $\epsilon>0$. If the above statement is not true, then one can find a sequence 
 $$(s_j,\theta_j)\in \C\times \R^m$$
 such that for all $j$ we have  $L_\Gamma(s_j,\theta_j)=0$ and $\mathrm{dist}(\theta_j,\Z^m)\geq\epsilon$ and $\lim_j s_j=\delta$.
 Using the $\Z^m$-periodicity of $L_\Gamma(s,\theta)$ with respect to $\theta$, we can assume that $\theta_j$ remains in a bounded subset of
 $\R^m$ and use compactness to extract a subsequence such that $\theta_j\rightarrow \ \widetilde{\theta}$ with 
 $\mathrm{dist}(\widetilde{\theta},\Z^m)\geq\epsilon$. We have $L_\Gamma( \delta,\widetilde{\theta})=0$, which is a contradiction with Proposition \ref{nonv}. \end{proof}

\subsection{Proof of Theorem \ref{main3}} 
We are now ready to prove Theorem \ref{main3}. We go back to the holomorphic map defined on $\C\times \C^m$
$$(s,\theta)\mapsto L_\Gamma(s,\theta).$$
Since $L_\Gamma(\delta,0)=0$ and we have (recall that $s=0$ is a simple zero of $Z_\Gamma(s)$)
$$\partial_s L_\Gamma(\delta,0)=Z'_\Gamma(\delta)\neq 0,$$
we can apply the Holomorphic implicit function theorem, which tells us that there exists an open set $\mathcal{O}\subset \C$ with $\delta \in \mathcal{O}$
and some $\epsilon>0$ such that for all $(s,\theta)\in \mathcal{O}\times B_\infty( 0,\epsilon)$,
$$ L_\Gamma(s,\theta)=0 \Longleftrightarrow s=\phi(\theta),$$
where $\phi:B_\infty( 0,\epsilon)\rightarrow \mathcal{O}$ is a real-analytic map and 
$$B_\infty( 0,\epsilon):=\{ x=(x_1,\ldots,x_r) \in \R^r\ :\ \max_\ell \vert x_\ell\vert < \epsilon\}.$$
Using Corollary \ref{smallz} with the above $\epsilon$, we deduce that if $s\in \mathcal{U}:=\mathcal{O}\cap \mathcal{V}$ is a such that
$$L_\Gamma(s,\theta)=0,$$
for some $\theta \in \R^m$, then $\mathrm{dist}(\theta,\Z^m)<\epsilon$, and $s=\phi(\widetilde{\theta})$
where $\widetilde{\theta}=\theta\ \mathrm{mod}\ \Z^m$ and $\widetilde{\theta}\in B_\infty( 0,\epsilon)$.

\bigskip  Now pick $\varphi \in C_0^\infty(\mathcal{U})$, using the factorization formula (\ref{factor1}), we observe that provided $\epsilon$ is taken small enough
we have
$$\sum_{\lambda \in \mathcal{R}_{X_j}\cap \mathcal{U}} \varphi(\lambda)
=\sum_{\vert k_1\vert <\epsilon N_1^{(j)},\ldots,\vert k_r\vert <\epsilon N_r^{(j)}}
\varphi \circ \phi \left(\frac{k_1}{N_1^{(j)}},\ldots,\frac{k_r}{N_r^{(j)}},0,\ldots,0 \right ).$$

Next we will apply the following Lemma.
\begin{lem}
 Fix $\epsilon>0$ and assume that $\psi$ is a $C^\infty(\R^r)$, compactly supported function on $B_\infty( 0,\epsilon)\subset \R^r$.
 Then we have,
 $$\lim_{j\rightarrow \infty}\frac{1}{N_1^{(j)}\ldots N_r^{(j)}}\sum_{\vert k_1\vert \leq \epsilon N_1^{(j)},\ldots,\vert k_r\vert \leq \epsilon N_r^{(j)}} 
 \psi\left(\frac{k_1}{N_1^{(j)}},\ldots,\frac{k_r}{N_r^{(j)}}\right)=\int_{B_\infty( 0,\epsilon)}\psi(x)dx.$$
\end{lem}

\begin{proof}
Use the Poisson summation formula to write
$$\frac{1}{N_1^{(j)}\ldots N_r^{(j)}}\sum_{\vert k_1\vert \leq \epsilon N_1^{(j)},\ldots,\vert k_r\vert \leq \epsilon N_r^{(j)}} 
 \psi\left(\frac{k_1}{N_1^{(j)}},\ldots,\frac{k_r}{N_r^{(j)}}\right)=\frac{1}{N_1^{(j)}\ldots N_r^{(j)}}\sum_{k\in \Z^r} 
 \psi\left(\frac{k_1}{N_1^{(j)}},\ldots,\frac{k_r}{N_r^{(j)}}\right)$$
 $$=\sum_{k\in \Z^r,k\neq 0}\widehat{\psi}(2\pi N_1k_1,\ldots,2\pi N_rk_r)+\int_{\R^r}\psi(x)dx,$$
 where $\widehat{\psi}$ is as usual the Fourier transform defined by
 $$\widehat{\psi}(\xi)=\int_{\R^r}\psi(x) e^{-i\xi.x}dx. $$
  Since $\widehat{\psi}$ has rapid decay (Schwartz class), a simple summation argument
 gives  
 $$\frac{1}{N_1^{(j)}\ldots N_r^{(j)}}\sum_{\vert k_1\vert \leq \epsilon N_1^{(j)},\ldots,\vert k_r\vert \leq \epsilon N_r^{(j)}} 
 \psi\left(\frac{k_1}{N_1^{(j)}},\ldots,\frac{k_r}{N_r^{(j)}}\right)$$
 $$=\int \psi(x)dx+O_\alpha \left (\frac{1}{(\min \{N_1^{(j)},\ldots N_r^{(j)} \})^\alpha } \right),$$
 for all integers $\alpha$, and the proof is done.
\end{proof}
 
Applying the above lemma with $\psi(x)=\varphi\circ \phi(x,0)$ we get as $j\rightarrow +\infty$,
$$\lim_{j\rightarrow +\infty} \frac{1}{\vert \G_j \vert} \sum_{\lambda \in \mathcal{R}_{X_j}\cap \mathcal{U}} \varphi(\lambda)=
N_{r+1}\ldots N_m \int_{\R^r} \varphi\circ \phi(x,0)dx$$
$$:=\int \varphi d\mu,$$
where $\phi(x,0)=\phi(x_1,\ldots,x_r,0,\ldots,0)$.

\bigskip 
The measure $\mu$ is nothing but the {\it push-forward of Lebesgue measure} on the ball $B_\infty(0,\epsilon)$ via the map $\phi$.
It is clear from the above formula that $\delta$ belongs to the support of $\mu$ since $\phi(0)=\delta$. 
What remains to show is:
\begin{itemize}
\item The maps $x\mapsto \phi(x,0)$ are real valued so that all the resonances in the vicinity of $s=\delta$ are actually real.
\item The maps $x\mapsto \phi(x,0)$ are non constant.
\item The corresponding push-forward measure $\mu$ is absolutely continuous.
\end{itemize}

Since all the resonances in $\mathcal{R}_{X_j}$ (also all zeros of $s\mapsto L_\Gamma(s,\theta)$ for $\theta \in \R^m$) are in the half plane $\{ \Re(s)\leq \delta \}$, we must have $\nabla \Re(\phi)(0)=0$. However, one can actually show that 
$\Im(\phi)=0$ identically. Indeed, recall that by using the same ideas as in $\S 2$, one can show that for $\Re(s)>\delta$ we have for all $\theta \in \Z^m$,
$$L_\Gamma(s,\theta)=\mathrm{exp}\left(-\sum_{\mathcal{C},k,n} \frac{\chi_\theta(\mathcal{C}^k)}{k}
e^{-(s+n)kl(\mathcal{C})} \right),$$
where the sum runs over prime conjugacy classes. By complex conjugation and uniqueness of analytic continuation, we have first the identity
valid for all $s\in \C$ and $\theta \in \Z^m$, 
$$\overline{L_\Gamma(s,\theta)}=L_\Gamma(\overline{s},-\theta),$$
which implies that for all $\theta \in B_\infty(0,\epsilon)$, we have 
$$\phi(-\theta)=\overline{\phi(\theta)}.$$ 
On the other hand, if $\mathcal{C}\in \mathcal{P}$, then $\mathcal{C}^{-1}\in \mathcal{P}$ and $l(\mathcal{C}^{-1})=l(\mathcal{C})$, while
$\chi_\theta( \mathcal{C}^{-1})=\chi_{-\theta}( \mathcal{C})$. Therefore "time reversal" invariance of $\mathcal{P}$ yields another identity (again use unique continuation) valid for all $s\in \C$ and $\theta \in \Z^m$,
$$L_\Gamma(s,\theta)=L_\Gamma(s,-\theta).$$
It shows that for all $\theta \in B_\infty(0,\epsilon)$, $\overline{\phi}(\theta)=\phi(-\theta)=\phi(\theta)$, hence $\phi$ is real valued. This fact was observed in previous works
related to prime orbit counting (in homology classes) for geodesics flows, see for example \cite[Chapter~12]{ParryPollicott2}. 
By the same arguments as above, we know that the Hessian matrix $\nabla^2 \Re(\phi)(0)$ must be negative. Because the zeta functions $Z_{\Gamma_j}(s)$ have all a simple zero at $s=\delta$, the maps $x\mapsto \phi(x,0)$ have to be non constant.

One can actually show, using that the length spectrum of $X$ is not a lattice, that (see for example the arguments in \cite[page~199]{ParryPollicott2}) we have
$$\det\left (\nabla^2 \Re(\phi)(0)\right)<0.$$
We point out that the non-deneneracy of this critical point has historically played an important role on works related to prime orbit counting
in homology classes, see \cite{Anantharaman, KatSun,Lalley2, PhillipsSarnak,Pollicott2}.
Since each map $(x_1,\ldots,x_r)\mapsto \phi(x_1,\dots,x_r,0)\in \R$ is non-constant, the (closure of the) image is a non-trivial interval of the type $I=[a,\delta]$ 
for some $a<\delta$. Moreover, because 
$$(x_1,\ldots,x_r)\mapsto F(x):=\phi(x_1,\dots,x_r,0)$$
is real analytic (and non-constant),
the set of points $x=(x_1,\dots,x_r)\in B_\infty(0,\epsilon)$ such $\nabla F(x)=0$ has zero lebesgue measure. It follows
from standard arguments (see for example in \cite{Pono}) that $F$ has the "$0$-set" property: the preimage of each set of
zero Lebesgue measure has zero Lebesgue measure. We can apply Radon-Nikodym  theorem and conclude
that $\mu$ is absolutely continuous with respect to Lebesgue on $I$, the proof is complete. 

It is possible to describe the Radon-Nikodym derivative $\frac{d\mu}{dm}(u)$ in the vicinity of $\delta$, 
where $m$ is Lebesgue measure on $I$. Indeed, we know from the above that locally, 
$$\phi(x)=\delta-Q(x)+O(\Vert x\Vert^3),$$ 
where $Q(x)$ is a positive definite quadratic form.

The Morse lemma implies that for all $ \epsilon > 0 $ small enough there an open neighbourhood $ \tilde{U}\subset \mathbb{R}^{r} $ of $ 0 $ and a diffeomorphism
$$
\Psi : B_{\infty}(0,\varepsilon) \to  \tilde{U}, \quad (x_{1},\dots, x_{r})\mapsto (y_{1},\dots, y_{r})
$$
such that $ \Psi(0)=0 $ and $ \phi\circ \Psi^{-1}(y) = \delta - y_{1}^{2}-\cdots-y_{r}^{2}. $ Therefore, for any $ \varphi\in C^{\infty}_{0}(\mathcal{U}) $ we have
\begin{align*}
\int \varphi d\mu &= \int_{\R^r} \varphi\circ \phi(x,0)dx\\
&= \int_{\tilde{U}} \varphi(\delta-y_{1}^{2}-\cdots-y_{r}^{2})\cdot \vert D\Psi^{-1}(y)\vert dy\\
& \asymp\int_{\tilde{U}} \varphi(\delta-y_{1}^{2}-\cdots-y_{r}^{2}) dy,
\end{align*}
where $ \vert D\Psi^{-1}(y)\vert $ is the Jacobian determinant. Choosing polar coordinates yields
$$ \int \varphi d\mu \asymp\int_{\mathbb{R}^{+}} \varphi(\delta-R^{2}) R^{r-1}dR. $$
With one last change of variables $ R\mapsto \xi= R^{2} $ we obtain
$$ \int \varphi d\mu \asymp\int_{\mathbb{R}^{+}} \varphi(\delta-\xi) \xi^{\frac{r-2}{2}}d\xi. $$
We conclude that there exists a constant $C>0$ such that for all $ u $ close enough to delta ($u<\delta$)
$$C^{-1}(\delta-u)^{\frac{r-2}{2}}\leq \frac{d\mu}{dm}(u)\leq C(\delta-u)^{\frac{r-2}{2}},$$
where $r$ is defined above as the number of unbounded cyclic factors in the sequence of abelian groups $\G_j$. 
In particular we observe a drastic difference in the density shape when $r=1,2$ and $r>2$.

We conclude this section on abelian covers by a remark on the case of elementary groups (which we have excluded so far). Given a non trivial hyperbolic isometry
$\gamma$ in $PSL_2(\R)$, we set $\Gamma=\langle \gamma \rangle$ and $X=\Gamma \backslash \hh$ the corresponding hyperbolic cylinder. It is easy to check that
all finite covers of $X$ are (obviously) abelian given by 
$$X_N=\Gamma_N \backslash \hh,\ \Gamma_N=\langle \gamma^N \rangle,$$
with $N\geq 1$. In that case, it is possible to compute explicitly (see Borthwick \cite{Borthwick} p. 179) the Selberg zeta function 
$$Z_{X_N}(s)=\prod_{k\geq 0}\left(1-e^{(s+k)Nl(\gamma)}\right)^2,$$
where $l(\gamma)$ is the length of $\gamma$. The zero-set of $Z_{X_N}(s)$ is therefore the half-lattice
$$\frac{2i\pi}{Nl(\gamma)}\Z-\N_0,$$
from which we can see that resonances accumulate as $N\rightarrow +\infty$ on the axis 
$$\{ \Re(s)=\delta=0\}.$$
Notice that resonances have multiplicity {\it two}, which explains why the perturbative argument used in the non elementary case doesn't work here.

\section{Zero-free regions for $L$-functions and explicit formulae}
The goal of this section is to prove the following result which will allow us to convert zero-free regions into upper bounds on sums
over closed geodesics. The results are completely general, but will be used in the last section on congruence subgroups.

\begin{propo}
\label{ExplicitF}
Fix $\overline{\alpha}>0$, $0\leq \sigma<\delta$ and $\varepsilon>0$. Then there exists a $C_0^\infty$ test function $\varphi_0$, with
$\varphi_0\geq 0$, $\mathrm{Supp}(\varphi_0)=[-1,+1]$ and such that that for $\rr$ non trivial, if $L_\Gamma(s,\rr)$ has no zeros in the rectangle
$$\{ \sigma \leq \Re(s) \leq 1\ \mathrm{and}\ \vert \Im(s)\vert \leq (\log T)^{1+\overline{\alpha}}\},$$
for some $T$ large enough, then we have 
$$\sum_{\mathcal{C},k} \chi_\rr(\mathcal{C}^k)
\frac{l(\mathcal{C})}{1-e^{kl(\mathcal{C})}}\varphi_0\left( \frac{kl(\mathcal{C})}{T}\right)=
O\left(d_\rr \log(d_\rr+1)e^{(\sigma+\varepsilon)T}\right),$$
where the implied constant is uniform in $T, d_\rr$.
\end{propo}
The proof will occupy the full section and will be broken into several elementary steps.

\bigskip
\subsection{Preliminary Lemmas}
We start this section by the following fact from harmonic analysis.
\begin{lem}
\label{Fourier1}
For all $\alpha>0$, there exists $C_1,C_2>0$ and a positive test function $\varphi_0 \in C_0^\infty(\R)$ with 
$\mathrm{Supp}(\varphi)=[-1,+1]$ such that for all $\vert \xi\vert \geq 2$, we have
$$\vert \widehat{\varphi_0}(\xi)\vert \leq C_1 e^{\vert \Im(\xi)\vert} \exp\left (-C_2\frac{\vert \Re(\xi)\vert}{(\log\vert \Re(\xi)\vert)^{1+\alpha}}   \right),$$
where $\widehat{\varphi_0}(\xi)$ is the Fourier transform, defined as usual by
$$ \widehat{\varphi_0}(\xi)=\int_{-\infty}^{+\infty} \varphi_0(x)e^{-ix\xi}dx. $$
\end{lem}

\begin{proof}
It is known from the Beurling-Malliavin multiplier Theorem, or the Denjoy-Carleman Theorem, that for
compactly supported test functions $\psi$, one cannot beat the Fourier decay rate ($\xi \in \R$, large)
$$ \vert \widehat{\psi}(\xi)\vert =O\left (\exp\left(-C\frac{\vert \xi\vert}{\log\vert \xi\vert}\right)  \right),$$
because this rate of Fourier decay implies quasi-analyticity (hence no compactly supported test functions). We refer the reader to
\cite[Chapter~5]{Katz} for more details. The above statement is definitely a {\it folklore} result. However since we need a precise control for complex valued
$\xi$ and couldn't find the exact reference for it, we provide an outline of the proof which follows closely the construction that one can find in \cite[Chapter~5, Lemma~2.7]{Katz}.

Let $(\mu_j)_{j\geq 1}$ be a sequence of positive numbers such that $\sum_{j=1}^\infty \mu_j=1$. For all $k\in \Z$, set
$$\varphi_N(k)=\prod_{j=1}^N \frac{\sin(\mu_j k)}{\mu_j k},\ \  \varphi(k)=\prod_{j=1}^\infty \frac{\sin(\mu_j k)}{\mu_j k}.$$
Consider the Fourier series given by
$$f(x):=\sum_{k \in \Z} \varphi(k) e^{ikx},\ \ f_N(x):=\sum_{k \in \Z} \varphi_N(k) e^{ikx},$$
then one can observe that by rapid decay of $\varphi(k)$, $f(x)$ defines a $C^\infty$ function on $[-2\pi,2\pi]$.
On the other hand, one can check that $f_N(x)$ converges uniformly to $f$ as $N$ goes to $\infty$ and that
$$f_N(x)=(g_1 \star g_2\star \ldots \star g_N)(x),$$
where $\star$ is the convolution product and each $g_j$ is given by
$$g_j(x):=\left \{ \frac{2\pi}{\mu_j}\ \mathrm{if}\ \vert x\vert \leq \mu_j \atop 0\ \mathrm{elsewhere}. \right.$$
From this observation one deduces that $f$ is positive and supported on $[-1,+1]$ since we assume
$$\sum_{j=1}^\infty \mu_j=1.$$
We now extend $f$ outside $[-1,+1]$ by zero and write by integration by parts and Schwarz inequality,
$$\vert \widehat{ f}(\xi)\vert \leq  \frac{e^{\vert \Im(\xi)\vert}}{\vert \Re(\xi)\vert^N}\Vert f^{(N)}\Vert_{L^2(-1,+1)}.$$
By Plancherel formula, we get
$$\Vert f^{(N)}\Vert_{L^2(-1,+1)}^2=\sum_{k \in \Z} k^{2N}(\varphi(k))^2\leq C \prod_{j=1}^{N+1}\mu_j^{-2},$$
where $C>0$ is some universal constant. Fixing $\epsilon>0$, we now choose
$$\mu_j=\frac{\widetilde{C}}{j(\log(1+j))^{1+\epsilon}},$$
where $\widetilde{C}$ is adjusted so that $\sum_{j=1}^\infty \mu_j=1$, and we get
$$\vert \widehat{ f}(\xi)\vert \leq  \frac{e^{\vert \Im(\xi)\vert}}{\vert \Re(\xi)\vert^N} (C_1)^N N! e^{N(1+\epsilon)\log\log(N)}.$$
Using Stirling's formula and choosing $N$ of size
$$N=\left [ \frac{\vert \Re(\xi)\vert}{(\log(\vert\Re(\xi)\vert)^{1+2\epsilon}} \right ] $$
yields (after some calculations) to
$$\vert \widehat{ f}(\xi)\vert \leq  O\left (e^{\vert \Im(\xi)\vert} e^{-C_2 \frac{\vert \Re(\xi)\vert}{(\log(\vert \Re(\xi)\vert)^{1+2\epsilon}}}  
\right),$$
and the proof is finished. 
\end{proof}

One can obviously push the above construction further below the threshold $\frac{\vert \xi \vert}{\log\vert \xi\vert}$
by obtaining decay rates of the type
$$\exp \left (-\frac{\vert \xi\vert}{\log\vert \xi \vert \log(\log\vert \xi \vert) \ldots (\log_{(n)}\vert \xi \vert)^{1+\alpha}}\right),$$
where $\log_{(n)}(x)=\log\log\ldots\log(x)$, iterated $n$ times. However this would only yield a very mild improvement to the main statement,
so we will content ourselves with the above lemma.

We continue with another result which will allow us to estimate the size of the log-derivative of $L_\Gamma(s,\rr)$ in a narrow rectangular
zero-free region. More precisely, we have the following:
\begin{propo}
\label{Derivative1}
Fix $\sigma<\delta$. For all $\epsilon>0$, there exist $C(\epsilon), R(\epsilon) >0$ such that for all $R\geq R(\epsilon)$,
if $L_\Gamma(s,\rr)$ ($\rr$ is non-trivial) has no zeros in the rectangle
$$\{ \sigma \leq \Re(s) \leq 1\ \mathrm{and}\ \vert \Im(s)\vert \leq R\},$$
then we have for all $s$ in the smaller rectangle
$$\{ \sigma+\epsilon \leq \Re(s) \leq 1\ \mathrm{and}\ \vert \Im(s)\vert \leq C(\epsilon)R\},$$
$$\left \vert \frac{L'_\Gamma(s,\rr)}{L_\Gamma(s,\rr)}   \right \vert \leq B(\epsilon) d_\rr \log(d_\rr+1) R^6.$$
\end{propo}

\begin{proof}
We will use Caratheodory's Lemma and take advantage of the {\it a priori} bound from Theorem \ref{main1}. More precisely,
our goal is to rely on this estimate (see Titchmarsh \cite[5.51]{Tits}).
\begin{lem}
\label{Bigtits}
 Assume that $f$ is a holomorphic function on a neighborhood of the closed disc $\overline {D}(0,r)$, then
 for all $r'<r$, we have
 $$ \max_{\vert z\vert\leq r'} \vert f'(z) \vert \leq \frac{8r}{(r-r')^2}\left ( 
 \max_{\vert z  \vert \leq r} \vert \Re(f(z))\vert+\vert f(0)\vert \right).$$
 \end{lem}
 First we recall that for all $\Re(s)>\delta$, $L_\Gamma(s,\rr)$ does not vanish and has a representation as
 $$L_\Gamma(s,\rr)=\mathrm{exp}\left(-\sum_{\mathcal{C},k} \frac{\chi_\rr(\mathcal{C}^k)}{k}
\frac{e^{-skl(\mathcal{C})}}{1-e^{kl(\mathcal{C})}} \right),$$
so that we get for all $\Re(s)\geq A>\delta$,
\begin{equation}
\label{Center1}
\left \vert \log\vert L_\Gamma(s,\rr)\vert \right \vert \leq C_A d_\rr,\ \ \left \vert \frac{L'_\Gamma(s,\rr)}{L_\Gamma(s,\rr)} \right \vert \leq  C'_A d_\rr
\end{equation}
where $C_A, C'_A>0$ are uniform constants on all half-planes 
$$\{\Re(s)\geq A>\delta \}.$$
We have simply used the prime orbit theorem and the trivial bound on characters of unitary representations:
 $$\vert\chi_\rr(g)\vert \leq d_\rr,\ \mathrm{for\ all}\  g\in \G.$$
 Let us now assume that $L_\Gamma(s,\rr)$ does not vanish on the rectangle
 $$\{ \sigma \leq \Re(s) \leq 1\ \mathrm{and}\ \vert \Im(s)\vert \leq R\}.$$ 
 Consider the disc $D(M,r)$ centered at $M$ and with radius $r$ where $M(\sigma,R)$ and $r(\sigma,R)$ are given by
 $$M(\sigma,R)=\frac{R^2}{2(1-\sigma)}+\frac{\sigma+1}{2};\ \ r(\sigma,R)=M(\sigma,R)-\sigma,$$
 see the figure below.
 
 \begin{center}
\includegraphics[scale=0.5]{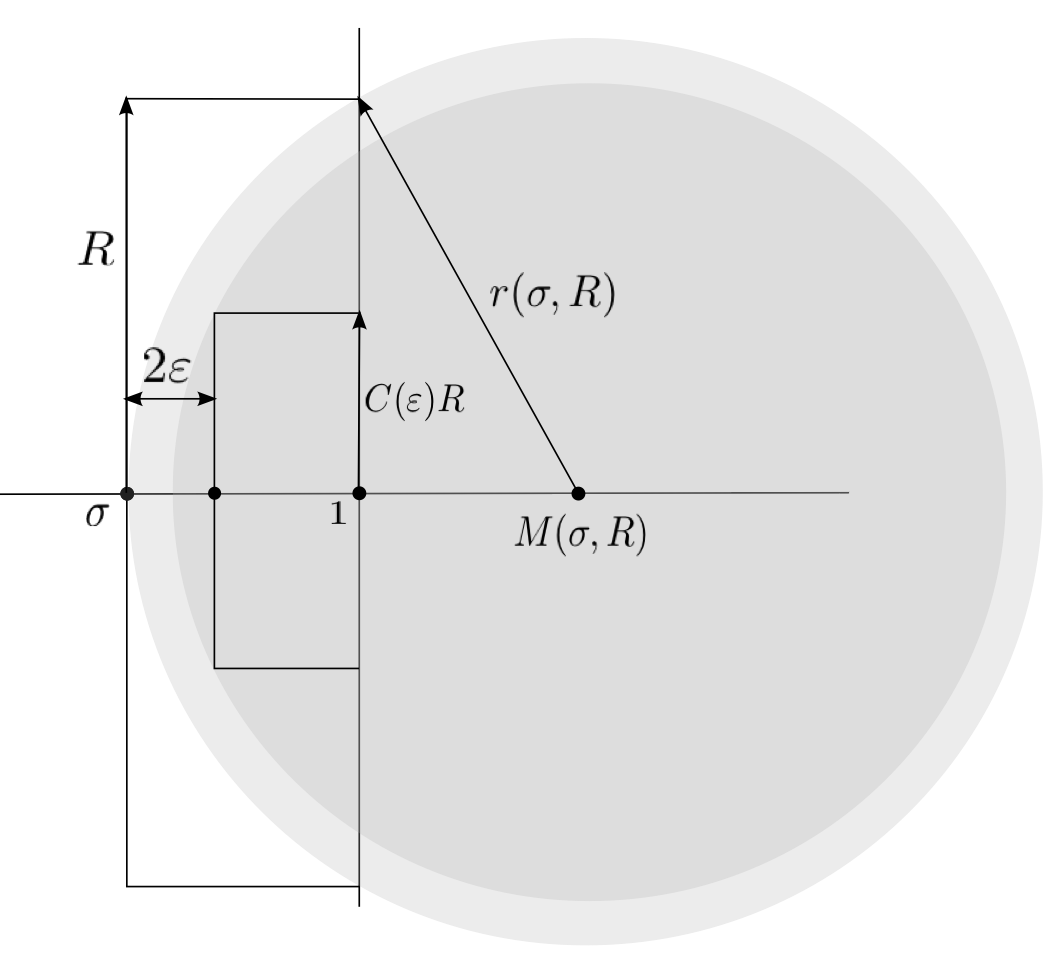}
\end{center}
Since by assumption $s\mapsto L_\Gamma(s,\rr)$ does not vanish on the closed disc $\overline{D}(M,r)$, we can choose a determination
of the complex logarithm of $L_\Gamma(s,\rr)$ on this disc to which we can apply Lemma \ref{Bigtits} on the smaller disc 
$D(M,r-\varepsilon)$, which yields (using the a priori bound from Theorem \ref{main1} and estimate (\ref{Center1}))
$$\left \vert \frac{L'_\Gamma(s,\rr)}{L_\Gamma(s,\rr)}   \right \vert
 \leq C\frac{r}{\varepsilon}\left (d_\rr \log(d_\rr+1) r^2 +A_1d_\rr \right )$$
 $$=O\left ( R^6 d_\rr \log(d_\rr+1) \right),$$
 where the implied constant is uniform with respect to $R$ and $d_\rr$.
 Looking at the picture, the smaller disc $D(M,r-\varepsilon)$ contains a rectangle
 $$\{ \sigma+2\varepsilon \leq \Re(s) \leq 1\ \mathrm{and}\ \vert \Im(s)\vert \leq L(\varepsilon)\},$$
 where $L(\varepsilon)$ satisfies the identity (Pythagoras Theorem!)
 $$L^2(\varepsilon)=\varepsilon(2M-2\sigma-3\varepsilon),$$
which shows that 
$$ L(\epsilon)\geq C(\varepsilon)R,$$
with $C(\varepsilon)>0$, as long as $R\geq R_0(\epsilon)$, for some $R_0>0$. The proof is done. 
\end{proof}

\subsection{Proof of Proposition \ref{ExplicitF}}
We are now ready to prove the main result of this section, by combining the above facts with a standard contour deformation argument. We fix a small $\varepsilon>0$ and $0<\alpha<\overline{\alpha}$. We use Lemma
\ref{Fourier1} to pick a test function $\varphi_0$ with Fourier decay as described, with same exponent $\alpha$. We set for all $T>0$, and $s\in \C$,
$$\psi_T(s)=\int_{-\infty}^{+\infty} e^{sx}\varphi_0\left(\frac{x}{T}\right)dx$$
$$=T\widehat{\varphi_0}(isT).$$
By the estimate from Lemma \ref{Fourier1}, 
we have 
\begin{equation}
\label{testdecay}
\vert \psi_T(s)\vert\leq C_1 Te^{T\vert \Re(s)\vert}\exp\left(-C_2 \frac{\vert \Im(s)\vert T}{(\log(T\vert \Im(s)\vert)^{1+\alpha}}\right).
\end{equation} 
We fix now $A>\delta$ and consider the contour integral
$$I(\rr,T)=\frac{1}{2i\pi} \int_{A-i\infty}^{A+i\infty} \frac{L'_\Gamma(s,\rr)}{L_\Gamma(s,\rr)}\psi_T(s)ds.$$
Convergence is guaranteed by estimate (\ref{Center1}) and rapid decay of $\vert \psi_T(s)\vert$ on vertical lines.
Because we choose $A>\delta$, we have absolute convergence of the series
$$\frac{L'_\Gamma(s,\rr)}{L_\Gamma(s,\rr)}=\sum_{\mathcal{C},k} \chi_\rr(\mathcal{C}^k)
\frac{l(\mathcal{C})e^{-skl(\mathcal{C})}}{1-e^{kl(\mathcal{C})}}$$
on the vertical line $\{\Re(s)=A\}$, 
and we can use Fubini to write
$$I(\rr,T)=\sum_{\mathcal{C},k} \chi_\rr(\mathcal{C}^k)
\frac{l(\mathcal{C})}{1-e^{kl(\mathcal{C})}}e^{-Akl(\mathcal{C})}\frac{1}{2\pi}\int_{-\infty}^{+\infty} e^{-itkl(\mathcal{C})}
\widehat{\psi_T}(iA-t)dt,$$
and Fourier inversion formula gives
$$I(\rr,T)=\sum_{\mathcal{C},k} \chi_\rr(\mathcal{C}^k)
\frac{l(\mathcal{C})}{1-e^{kl(\mathcal{C})}}\varphi_0\left( \frac{kl(\mathcal{C})}{T}\right).$$
Assuming that $L_\Gamma(s,\rr)$ has no zeros in 
$$\{ \sigma \leq \Re(s) \leq 1\ \mathrm{and}\ \vert \Im(s)\vert \leq R\},$$
where $R$ will be adjusted later on, our aim is to use Proposition \ref{Derivative1} to deform the contour integral $I(\rr,T)$ as depicted in the figure
below. 
 \begin{center}
\includegraphics[scale=0.6]{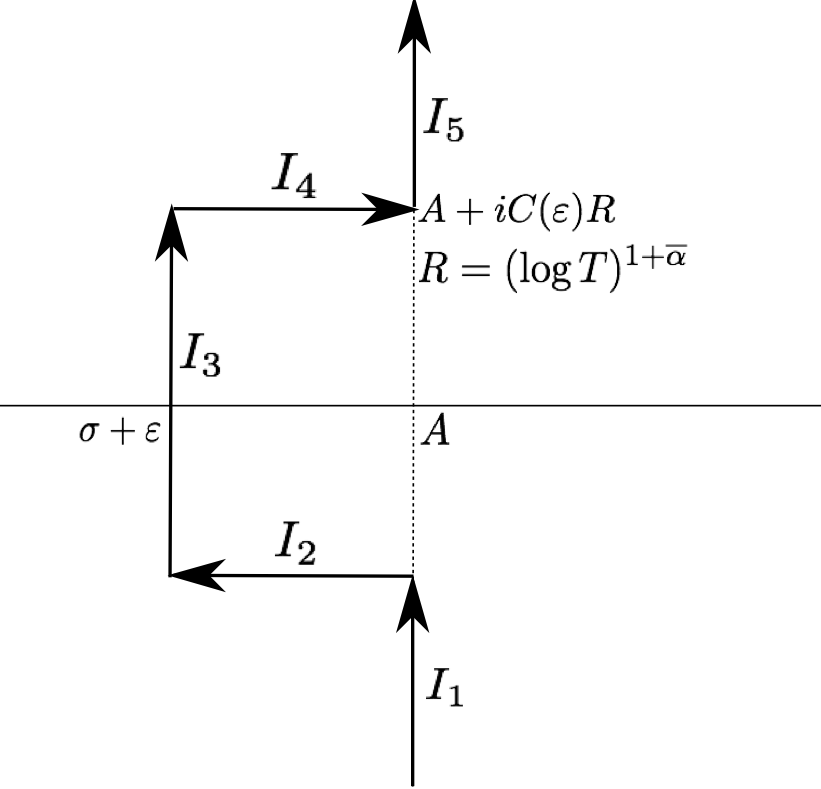}
\end{center} 
 Writing $I(\rr,T)=\sum_{j=1}^5 I_j$ (see the above figure), we need to estimate carefully each contribution. In the course of the proof, we will use the following basic fact.
 \begin{lem}
 \label{expo1}
 Let $\phi:[M_0,+\infty)\rightarrow \R^+$ be a $C^2$ map with $\phi'(x)>0$ on $[M_0,+\infty)$ and satisfying
 $$(*)\ \ \ \ \sup_{x\geq M_0} \left \vert \frac{\phi''(x)}{(\phi'(x))^3}\right \vert \leq C, $$
 then we have for all $M\geq M_0$,
 $$\int_M^{+\infty} e^{-\phi(t)}dt\leq \frac{e^{-\phi(M)}}{\phi'(M)}+Ce^{-\phi(M)}.$$
 \end{lem}

\begin{proof}
First observe that condition $(*)$ implies that 
 $$x\mapsto \frac{1}{(\phi'(x))^2}$$
 has a uniformly bounded derivative, which is enough to guarantee that 
 $$\lim_{x\rightarrow +\infty} \frac{e^{-\phi(x)}}{\phi'(x)}=0. $$
 In particular $\lim_{x\to +\infty} \phi(x)=+\infty$ and for all $M\geq M_0$, $\phi:[M,+\infty)\rightarrow [\phi(M),+\infty)$ is a $C^2$-diffeomorphism.
 A change of variable gives
 $$\int_M^{+\infty} e^{-\phi(t)}dt=\int_{\phi(M)}^{+\infty} e^{-u}\frac{du}{\phi'(\phi^{-1}(u))},$$
 and integrating by parts yields the result. 
\end{proof}

\bigskip

\begin{itemize}
\item First we start with $I_1$ and $I_5$. Using estimate (\ref{Center1}) combined with (\ref{testdecay}), we have
$$\vert I_5 \vert \leq C d_\rr T e^{TA} \int_{C(\varepsilon)R} ^{+\infty} e^{-C_2 \frac{tT}{(\log(tT))^{1+\alpha}}}dt,$$
which by a change of variable leaves us with
$$\vert I_5 \vert \leq C d_\rr e^{TA} \int_{C(\varepsilon)RT}^{+\infty} e^{-C_2 \frac{u}{(\log(u))^{1+\alpha}}}du.$$
This where we use Lemma \ref{expo1} with 
$$\phi(x)=C_2 \frac{x}{(\log(x))^{1+\alpha}}.$$
Computing the first two derivatives, we can check that condition $(*)$ is fulfilled and therefore
$$\int_{M}^{+\infty} e^{-C_2 \frac{u}{(\log(u))^{1+\alpha}}}\leq C (\log(M))^{1+\alpha} e^{-C_2 \frac{M}{(\log(M))^{1+\alpha}}},$$
for some universal constant $C>0$. We have finally obtained
$$ \vert I_5 \vert \leq C d_\rr e^{TA}(\log(RT))^{1+\alpha} e^{-C_2 \frac{RT}{(\log(RT))^{1+\alpha}}}.$$
Choosing $R=(\log(T))^{1+\overline{\alpha}}$, with $\overline{\alpha}>\alpha$ gives 

$$\vert I_5 \vert =O\left(d_\rr e^{TA}(\log(T))^{1+\alpha} e^{-C_2 T(\log(T))^{\overline{\alpha}-\alpha}}\right)$$
$$=O\left ( d_\rho e^{-BT}\right),$$
where $B>0$ can be taken as large as we want. The exact same estimate is valid for $I_1$.
\item The case of $I_4$ and $I_2$. Here we use the bound from Proposition \ref{Derivative1} and again (\ref{testdecay}) to get
$$ \vert I_4\vert+\vert I_2\vert=O\left ( d_\rr \log(d_\rr+1) e^{-BT} \right),$$
where $B$ can be taken again as large as we want.
\item We are left with $I_3$ where 
$$I_3=\frac{1}{2\pi}\int_{-C(\epsilon)R}^{+C(\epsilon)R} \frac{L'_\Gamma(\sigma+\varepsilon+it,\rr)}{L_\Gamma(\sigma+\varepsilon+it,\rr)}
\psi_T(\sigma+\varepsilon+it)dt.$$
Using Proposition \ref{Derivative1} and (\ref{testdecay}) we get
$$\vert I_3 \vert =O\left( d_\rr \log(d_\rr+1) (\log(T))^{7(1+\overline{\alpha})} e^{(\sigma+\epsilon)T}\right).$$
\end{itemize}
Clearly the leading term in the contour integral is provided by $I_3$, and the proof of Proposition \ref{ExplicitF} is now complete. 

\bigskip
We conclude this section by a final observation. If $\rr=id$ is the {\it trivial representation}, then $L_\Gamma(s,id)=Z_\Gamma(s)$ has a zero
at $s=\delta$, thus the best estimate for the contour integral $I(id,T)$ is given by (\ref{Center1}) and (\ref{testdecay}) which yields (by a change of variable)
$$\vert I(id,T) \vert \leq C_Ad_\rr \int_{-\infty}^{+\infty} \vert \psi_T(A+it)\vert dt \leq \widetilde{C}_A d_\rr Te^{TA}
\int_{-\infty}^{+\infty}\exp\left(-C_2 \frac{\vert  t\vert T}{\left(\log(T\vert t\vert)\right)^{1+\alpha}}\right)dt$$
$$=O\left(d_\rr e^{TA}\right).$$
Since $d_\rr=1$ and $A$ can be taken as close to $\delta$ as we want, the contribution from the trivial representation is of size
\begin{equation}
\label{trivial}
I(id,T)=O\left(e^{(\delta+\epsilon)T} \right).
\end{equation}

\section{Congruence subgroups and existence of "low lying" zeros for $L_\Gamma(s,\rr)$}
\subsection{Conjugacy classes in $\G$.}
In this section, we will use more precise knowledge on the group structure of 
$$\G=SL_2(\Fp).$$
Our basic reference is the book \cite{Suzuki}, see Chapter 3, $\S 6$ for much more general statements over finite fields.
We start by describing the conjugacy classes in $\G$. Since we are only interested in the large $p$ behaviour, we will assume that
$p$ is an odd prime strictly bigger than $3$. Conjugacy classes of elements $g\in \G$ are essentially determined by the roots of the characteristic polynomial
$$\det(xI_2-g)=x^2-\mathrm{tr}(g)x+1,$$
which are denoted by $\lambda, \lambda^{-1}$, where $\lambda \in \Fp^{\times}$. There are three different possibilities.
\begin{itemize}
\item $\lambda\neq \lambda^{-1} \in  \Fp^{\times}$. In that case $g$ is diagonalizable over $\Fp$ and $g$ is conjugate to the matrix
$$D(\lambda)=\left (\begin{array}{cc}
\lambda &0\\
0&\lambda^{-1}
\end{array} \right ).$$
The centralizer $\mathcal{Z}(D(\lambda))=\{ h\in \G\ :\ hD(\lambda)h^{-1}=D(\lambda)\}$ is then equal to the "maximal torus"
$$A=\left\{  \left (\begin{array}{cc}
a &0\\
0&a^{-1}
\end{array} \right )\ :\ a \in   \Fp^{\times}\right\},$$
and we have $\vert A\vert=p-1$, the conjugacy class of $g$ has $p(p+1)$ elements.
\item $\lambda\neq \lambda^{-1} \not \in  \Fp^{\times}$. In that case $\lambda$ belongs to $\mathcal{F}\simeq \mathbb{F}_{p^2}$ the unique quadratic extension of $\Fp$. The root $\lambda$ can be written as 
$$ \lambda=a+b\sqrt{\epsilon},\  \lambda^{-1}=a-b\sqrt{\epsilon},$$
where $\{1,\sqrt{\epsilon}\}$ is a fixed $\Fp$-basis of $\mathcal{F}$.
Therefore $g$ is conjugate to 
$$\left (\begin{array}{cc}
a &\epsilon b\\
b& a
\end{array} \right),$$
and $\vert \mathcal{Z}(g)\vert=p+1$, its conjugacy class has $p(p-1)$ elements.
\item $\lambda=\lambda^{-1}\in \{\pm 1\}$. In that case $g$ is non-diagonalizable unless $g\in \mathcal{Z}(\G)=\{\pm I_2\}$, and is conjugate to
$\pm u$ or $\pm u'$ where 
$$u= \left (\begin{array}{cc}
1 & 1\\
0& 1
\end{array} \right),\ u'=\left (\begin{array}{cc}
1 & \epsilon\\
0& 1
\end{array} \right).$$
The centralizer $\mathcal{Z}(g)$ has cardinality $2p$ and the four conjugacy classes have $p(p+1)$ elements.  
\end{itemize}
Using this knowledge on conjugacy classes, one can construct all irreducible representations and write a character table for $\G$,
but we won't need it. There are two facts that we highlight and will use in the sequel:

\begin{enumerate}
 \item For all $g \in \G$, $\vert \mathcal{Z}(g)\vert \geq p-1$.
 \item For all $\rr$ non-trivial we have $d_\rho\geq \frac{p-1}{2}$.
\end{enumerate}
We will also rely on the very important observation below.
\begin{propo}
\label{Conj1}
Let $\Gamma$ be a convex co-compact subgroup of $SL_2(\Z)$ as above.
Fix $0<\beta<2$, and consider the set $\mathcal{E}_T$ of conjugacy classes $\overline{\gamma}\subset \Gamma\setminus \{Id\}$ such that for all 
$\overline{\gamma}\in \mathcal{E}_T$, we have $\l(\gamma)\leq T:=\beta \log(p)$. Then for all $p$ large and all
$\overline{\gamma_1},\overline{\gamma_2}\in \mathcal{E}_T$, the following are equivalent:
\begin{enumerate}
\item $\mathrm{tr}(\gamma_1)=\mathrm{tr}(\gamma_2)$.
\item $\gamma_1$ and $\gamma_2$ are conjugate in $\G$.
\end{enumerate}
\end{propo}

\begin{proof}
Clearly $(1)$ implies that $\gamma_1$ and $\gamma_2$ have the same trace modulo $p$. Unless
we are in the cases $\mathrm{tr}(\gamma_1)=\mathrm{tr}(\gamma_2)=\pm 2\ \mathrm{mod}\ p$, we know from the above description of conjugacy classes that
they are determined by the knowledge of the trace.  To eliminate these "parabolic mod $p$" cases, we observe that if $\overline{\gamma}\in \mathcal{E}_T$ satisfies 
$\mathrm{tr}(\gamma)=\pm 2 +kp$ with $k\neq 0$, then 
$$2\cosh(l(\gamma)/2)=\vert \mathrm{tr}(\gamma)\vert\geq p-2,$$
and we get
$$p-2\leq 1+p^{\frac{\beta}{2}},$$
which leads to an obvious contradiction if $p$ is large, therefore $k=0$.  Then it means that 
$\vert\mathrm{tr}(\gamma)\vert=2$ which is impossible since $\Gamma$ has no non trivial parabolic element (convex co-compact hypothesis). Conversely,
if $\gamma_1$ and $\gamma_2$ are conjugate in $\G$, then we have 
$$\mathrm{tr}(\gamma_1)=\mathrm{tr}(\gamma_2)\ \mathrm{mod}\ p.$$
If $\mathrm{tr}(\gamma_1)\neq \mathrm{tr}(\gamma_2)$ then this gives
$$ p\leq \vert \mathrm{tr}(\gamma_1)-\mathrm{tr}(\gamma_2)\vert \leq 4\cosh(T/2)\leq 2(p^{\frac{\beta}{2}}+1),$$
again a contradiction for $p$ large. 
\end{proof}

\subsection{Proof of Theorem \ref{main2}} 
Before we can rigourously prove Theorem \ref{main2}, we need one last fact from representation theory which is a handy folklore formula.
\begin{lem}
\label{Product1}
Let $\G$ be a finite group and let $\rr:\G\rightarrow \mathrm{End}(V_\rr)$ be an irreducible representation.
Then for all $x,y\in \G$, we have
$$\chi_\rr(x)\overline{\chi_\rr(y)}=\frac{d_\rr}{\vert G\vert} \sum_{g\in \G} \chi_\rr(xgy^{-1}g^{-1}).$$
\end{lem}

\begin{proof}
Writing
$$ \sum_{g\in \G} \chi_\rr(xgy^{-1}g^{-1})=\mathrm{Tr}\left(\rr(x) \sum_{g} \rr(gy^{-1}g^{-1}) \right),$$
we observe that 
$$U_y:=\sum_{g} \rr(gy^{-1}g^{-1})$$
commutes with the irreducible representation $\rr$, therefore by Schur's Lemma \cite[Chapter~2]{serre}, it has to be of the form
$$U_y=\lambda(y)I_{V_\rr},$$
with $\lambda(y)\in \C$, which shows that
$$\sum_{g\in \G} \chi_\rr(xgy^{-1}g^{-1})=\chi_\rr(x)\lambda(y).$$
Similarly we obtain
$$ \sum_{g\in \G} \chi_\rr(xgy^{-1}g^{-1})=\overline{\chi_\rr(y)}\lambda(x),$$
and evaluating at the neutral element $x=e_\G$ ends the proof since we have
$$U_{e_\G}=\vert \G \vert I_{V_\rr}. $$
\end{proof}

\bigskip We fix some $0\leq \sigma<\delta$. We take $\varepsilon>0$ and $\alpha>0$. We assume that for all non-trivial representation $\rr$, the corresponding $L$-function $L_\Gamma(s,\rr)$ does not vanish on the rectangle 
$$\{ \sigma \leq \Re(s) \leq 1\ \mathrm{and}\ \vert \Im(s)\vert \leq (\log T)^{1+\alpha}\},$$ 
where $T=\beta \log(p)$ with $0<\beta<2$. The idea is to look at the average
$$S(p):=\sum_{\rr\  \mathrm{irreducible}} \vert I(\rr,T) \vert^2,$$
where $I(\rr,T)$ is the sum given by
$$I(\rr,T)=\sum_{\mathcal{C},k} \chi_\rr(\mathcal{C}^k)
\frac{l(\mathcal{C})}{1-e^{kl(\mathcal{C})}}\varphi_0\left( \frac{kl(\mathcal{C})}{T}\right).$$
While each term $I(\rr,T)$ is hard to estimate from below because of the oscillating behaviour of characters, the mean square is tractable thanks to
Lemma \ref{Product1}. Let us compute $S(p)$.
$$S(p)=\sum_{\rr\  \mathrm{irreducible}} \sum_{\mathcal{C},k} \sum_{\mathcal{C'},k'} 
\frac{l(\mathcal{C})l(\mathcal{C'})}{(1-e^{kl(\mathcal{C})})(1-e^{k'l(\mathcal{C'})})}
\varphi_0\left( \frac{kl(\mathcal{C})}{T}\right)\varphi_0\left( \frac{k'l(\mathcal{C'})}{T}\right)
\chi_\rr(\mathcal{C}^k)\overline{\chi_\rr(\mathcal{C'}^{k'})}.$$
Using Lemma \ref{Product1}, we have
$$\chi_\rr(\mathcal{C}^k)\overline{\chi_\rr(\mathcal{C'}^{k'})}=
\frac{d_\rr}{\vert G\vert} \sum_{g\in \G} \chi_\rr(\mathcal{C}^kg(\mathcal{C'})^{-k'}g^{-1}),$$
and Fubini plus the identity
$$\sum_{\rr\  \mathrm{irreducible}} d_\rr \chi_\rr(g)=\vert \G \vert \mathcal{D}_e(g)$$
allow us to obtain
$$S(p)= \sum_{\mathcal{C},k} \sum_{\mathcal{C'},k'} 
\frac{l(\mathcal{C})l(\mathcal{C'})}{(1-e^{kl(\mathcal{C})})(1-e^{k'l(\mathcal{C'})})}
\varphi_0\left( \frac{kl(\mathcal{C})}{T}\right)\varphi_0\left( \frac{k'l(\mathcal{C'})}{T}\right)
\Phi_\G(\mathcal{C}^k,\mathcal{C'}^{k'}),$$
where
$$\Phi_\G(\mathcal{C}^k,\mathcal{C'}^{k'}):=\sum_{g\in \G} \mathcal{D}_e(\mathcal{C}^kg(\mathcal{C'})^{-k'}g^{-1}).$$
Since all terms in this sum are now positive and $\mathrm{Supp}(\varphi_0)=[-1,+1]$, we can fix a small $\varepsilon>0$ and find a constant $C_\varepsilon>0$ such that
$$S(p)\geq C_\varepsilon  \sum_{kl(\mathcal{C})\leq {\tiny T(1-\varepsilon)} \atop k'l(\mathcal{C'})\leq T(1-\varepsilon)} \Phi_\G(\mathcal{C}^k,\mathcal{C'}^{k'}).$$
Observe now that
$$\Phi_\G(\mathcal{C}^k,\mathcal{C'}^{k'})=\sum_{g\in \G} \mathcal{D}_e(\mathcal{C}^k g(\mathcal{C'})^{-k'}g^{-1})\neq 0$$
if and only if $\mathcal{C}^k$ and $\mathcal{C'}^{-k'}$ {\it are in the same conjugacy class mod $p$}, and in that case,
$$\Phi_\G(\mathcal{C}^k,\mathcal{C'}^{k'})=\vert \mathcal{Z}(\mathcal{C}^k)\vert=\vert \mathcal{Z}(\mathcal{C'}^{k'})\vert.$$
Using the lower bound for the cardinality of centralizers, we end up with
$$S(p)\geq C_\varepsilon (p-1) \sum_{\overline{\mathcal{C}^k}=\overline{\mathcal{C'}^{k'}}\ \mathrm{mod}\ p
\atop kl(\mathcal{C}), k'l(\mathcal{C'})\leq {\tiny T(1-\varepsilon)}} 1.$$
Notice that since we have taken $T=\beta\log(p)$ with $\beta<2$, we can use Proposition \ref{Conj1} which says that 
$\mathcal{C}^k$ and $\mathcal{C'}^{-k'}$ are in the same conjugacy class mod $p$ iff they have the same traces (in $SL_2(\Z)$).
It is therefore natural to rewrite the lower bound for $S(p)$ in terms of traces. We need to introduce a bit more notations. Let 
$\mathscr{L}_\Gamma$ be set of traces i.e.
$$\mathscr{L}_\Gamma=\{ \mathrm{tr}(\gamma)\ :\ \gamma \in \Gamma \}\subset \Z.$$
Given $t\in \mathscr{L}_\Gamma$, we denote by $m(t)$ the multiplicity of $t$ in the trace set by
$$m(t)=\#\{ \mathrm{conj\ class\ }\overline{\gamma}\subset\Gamma\ :\ \mathrm{tr}(\gamma)=t \}.$$
We have therefore (notice that multiplicities are squared in the double sum)
$$S(p)\geq C_\varepsilon (p-1) \sum_{t\in \mathscr{L}_\Gamma \atop \vert t \vert\leq {2\cosh(T(1-\varepsilon)/2)} } m^2(t). $$
To estimate from below this sum, we use a trick that goes back to Selberg. By the prime orbit theorem \cite{Naud3, Lalley,Roblin} applied to
the surface $\Gamma \backslash \hh$, we know that for all $T$ large, we have
$$Ce^{(\delta-2\varepsilon)T}\leq \sum_{t\in \mathscr{L}_\Gamma \atop \vert t \vert\leq {2\cosh(T(1-\varepsilon)/2)} } m(t),$$
and by Schwarz inequality we get for $T$ large
$$Ce^{(\delta-2\varepsilon)T}\leq C_0
\left ( \sum_{t\in \mathscr{L}_\Gamma \atop \vert t \vert\leq {2\cosh(T(1-\varepsilon)/2)} } m^2(t)\right)^{1/2} e^{T/4},$$
where we have used the obvious bound
$$\#\{ n\in \Z\ :\ \vert n\vert \leq 2\cosh(T(1-\varepsilon)/2)\}=O(e^{T/2}).$$
This yields the lower bound
$$\sum_{t\in \mathscr{L}_\Gamma \atop \vert t \vert\leq {2\cosh(T(1-\varepsilon)/2)} } m^2(t)\geq C_\varepsilon'  e^{(2\delta-1/2-\varepsilon)T},$$
which shows that one can take advantage of exponential multiplicities in the length spectrum when $\delta>\half$, thus beating the simple bound coming from the prime orbit theorem.
In a nutshell, we have reached the lower bound (for all $\varepsilon>0$), 
$$S(p)\geq C_\varepsilon (p-1)e^{(2\delta-1/2-\varepsilon)T}.$$
Keeping that lower bound in mind, we now turn to upper bounds using Proposition \ref{ExplicitF}.
Writing
$$S(p)=\vert I(id,T)\vert^2 +\sum_{\rr\neq id} \vert I(\rr,T)\vert^2, $$
and using the bound (\ref{trivial}) combined with the conclusion of Proposition \ref{ExplicitF}, we get 
$$S(p)=O(e^{(2\delta+\varepsilon)T})+O\left (  \sum_{\rr\neq id} d_\rr^2 (\log(d_\rr+1))^2 e^{2(\sigma+\varepsilon)T}  \right).$$
Using the formula
$$\vert \G \vert =\sum_{\rr} d_\rr^2,$$
combined with the fact that $\vert \G \vert=p(p^2-1)=O(p^3)$, 
we end up with
$$S(p)=O(e^{(2\delta-\varepsilon)T})+O\left (p^3\log(p) e^{2(\sigma+\varepsilon)T}  \right). $$
Since $T=\beta \log(p)$, we have obtained for all $p$ large \footnote{Note that the $\log(p)$ term has been absorbed in $p^\varepsilon$.}
$$Cp^{(2\delta-1/2-\varepsilon)\beta}\leq p^{(2\delta+\varepsilon)\beta-1}+p^{2+2(\sigma+\varepsilon)\beta+\varepsilon}.$$ 
Remark that since $\beta<2$, then if $\varepsilon$ is small enough we always have
$$(2\delta+\varepsilon)\beta-1< (2\delta-1/2-\varepsilon)\beta,$$
so up to a change of constant $C$, we actually have for all large $p$
$$Cp^{(2\delta-1/2-\varepsilon)\beta}\leq p^{2+2(\sigma+\varepsilon)\beta+\varepsilon}.$$
We have contradiction for $p$ large provided
$$\sigma< (\delta-\frac{1}{4}-\frac{1}{\beta})-\varepsilon-\frac{\varepsilon}{2\beta}.$$
Since $\beta$ can be taken arbitrarily close to $2$ and $\varepsilon$ arbitrarily close to $0$, we have a contradiction whenever
$\delta>\frac{3}{4}$ and $\sigma<\delta-\frac{3}{4}$. Therefore for all $p$ large, at least one of the $L$-function $L_\Gamma(s,\rr)$ 
for non trivial $\rr$ has to vanish inside the rectangle
$$\left \{ \delta-\textstyle{\frac{3}{4}}-\epsilon\leq \Re(s)\leq \delta\ \mathrm{and}\ 
\vert \Im(s) \vert \leq \left (\log(\log(p))\right)^{1+\alpha}\right \},$$
but then by the product formula we know that this zero appears as a zero of $Z_{\Gamma(p)}(s)$ with multiplicity $d_\rho$ which is greater or equal to $\frac{p-1}{2}$ by Frobenius. The main theorem is proved. \QEDB

\bigskip
We end by a few comments. 
It would be interesting to know if the $\log^{1+\epsilon}(\log(p))$ bound can be improved to a uniform constant. However, it would likely
require a completely different approach since $\log(\log(p))$ is the very limit one can achieve with compactly supported test functions.
Indeed, to achieve a uniform bound with our approach would require the use of test functions $\varphi \not \equiv 0$ with Fourier
bounds 
$$\vert \widehat{\varphi}(\xi)\vert \leq C_1 e^{\vert \Im(\xi)\vert} e^{-C_2\vert \Re(\xi)\vert},$$ 
but an application of the Paley-Wiener theorem shows that these test functions do not exist (they would be both compactly supported
and analytic on the real line). 

\section{Fell's continuity and Cayley graphs of abelian groups}
In this section we prove Theorem \ref{main4}. The arguments follow closely those of Gamburd in \cite{Gamburd1}. Roughly speaking,
since Cayley graphs of finite Abelian groups can never form a family of expanders, one should expect strongly that there is no uniform spectral gap
in the family of covers $X_j=\Gamma_j \backslash \hh$. We give a rigourous proof of that fact using Fell's continuity.

Let $ \mathcal{G} $ be a finite graph with set of vertices $ \mathcal{V} $ and of degree $ k $. That is, for every vertex $ x\in \mathcal{V} $ there are $ k $ edges adjacent to $ x $. For a subset of vertices $ A\subset \mathcal{V} $ we define its boundary $ \partial A $ as the set of edges with one extremity in $ A $ and the other in $ \mathcal{G}-A. $ The Cheeger isoperimetric constant $ h(\mathcal{G}) $ is defined as
$$ h(\mathcal{G}) = \min \left\{  \frac{\vert \partial A \vert}{\vert A\vert} : A\subset \mathcal{V} \text{ and } 1 \leq \vert A\vert \leq \frac{\vert \mathcal{V}\vert}{2}\right\}. $$
Let $ L^{2}(\mathcal{V}) $ be the Hilbert space of complex-valued functions on $ \mathcal{V} $ with inner product
$$ \langle F, G \rangle_{L^{2}(\mathcal{V})} = \sum_{x\in \mathcal{V}} F(x) \overline{G(x)}. $$
Let $ \Delta $ be the discrete Laplace operator acting on $ L^{2}(\mathcal{V}) $ by
$$ \Delta F(x) = F(x) - \frac{1}{k}\sum_{y\sim x} F(y), $$
where $ F\in L^{2}(\mathcal{V}) $, $ x\in \mathcal{V} $ is a vertex of $ \mathcal{G} $, and $ y\sim x $ means that $ y $ and $ x $ are connected by an edge. The operator $ \Delta $ is self-adjoint and positive. Let $ \lambda_{1}(\mathcal{G}) $ denote the first non-zero eigenvalue of $ \Delta. $ 

The following result due to Alon and Milman \cite{AlonMilman} relates the spectral gap $ \lambda_{1}(\mathcal{G}) $ and Cheeger's isoperimetric constant.

\begin{propo}\label{Spectral Gap and Cheeger}
For finite graphs $ \mathcal{G} $ of degree $ k $ we have
$$
\frac{1}{2}k\cdot \lambda_{1}(\mathcal{G}) \leq h(\mathcal{G}) \leq k \sqrt{\lambda_{1}(\mathcal{G})(1-\lambda_{1}(\mathcal{G}))}.
$$
\end{propo}

We note that large first non-zero eigenvalue $ \lambda_{1}(\mathcal{G}) $ implies fast convergence of random walks on $ \mathcal{G} $, that is, high connectivity (see Lubotzky \cite{Lubotzky2}).

\begin{defi}
A family of finite graphs $ \{ \mathcal{G}_{j} \} $ of bounded degree is called a family of expanders if there exists a constant $ c>0 $ such that $ h(\mathcal{G}_{j})\geq c. $
\end{defi}

The family of graphs we are interested in is built as follows. Let $ \Gamma=\langle S\rangle $ be a Fuchsian group generated by a finite set $ S\subset \mathrm{PSL}_{2}(\mathbb{R}) $. We will assume that $ S $ is symmetric, i.e. $ S^{-1}=S. $ Given a sequence $ \Gamma_{j} $ of finite index normal subgroups of $ \Gamma $, let $ S_{j} $ be the image of $ S $ under the natural projection $ r_{\textbf{G}_{j}} : \Gamma\to \textbf{G}_{j}=\Gamma/\Gamma_{j} $. Notice that $ S_{j} $ is a symmetric generating set for the group $ \textbf{G}_{j} $. Let $ \mathcal{G}_{j} = \mathrm{Cay}(\textbf{G}_{j}, S_{j}) $ denote the Cayley graph of $ \textbf{G}_{j} $ with respect to the generating set $ S_{j} $. That is, the vertices of $ \textbf{G}_{j} $ are the elements of $ \textbf{G}_{j} $ and two vertices $ x $ and $ y $ are connected by an edge if and only if $ xy^{-1}\in S_{j}. $

The connection of uniform spectral gap with the graphs constructed above comes from the following result.

\begin{propo}\label{expanders and gap}
Assume that $ \delta=\delta(\Gamma) > \frac{1}{2} $ and assume that there exists $ \epsilon > 0 $ such that for all $ j $ all non trivial resonances $ s $ of $ X_{j}=\Gamma_{j}\setminus \mathbb{H} $ satisfy $ \vert s-\delta\vert > \epsilon. $ Then the Cayley graphs $ \mathcal{G}_{j} $ form a family of expanders.
\end{propo}

Let us see how Proposition \ref{expanders and gap} implies Theorem \ref{main4}.

\begin{proof}[Proof of Theorem \ref{main4}]
Since $ X=\Gamma\setminus \mathbb{H} $ has at least one cusp by assumption, we have $ \delta > \frac{1}{2} $ so that we can apply Proposition \ref{expanders and gap}. Suppose by contradiction that there exists $ \epsilon > 0 $ such that for all $ j $ we have $ \vert s-\delta\vert > \epsilon $ for all non trivial resonances $ s $ of $ X_{j} $. Then Proposition \ref{expanders and gap} implies that the Cayley graphs $ \mathcal{G}_{j} = \mathrm{Cay}(\textbf{G}_{j}, S_{j}) $ form a family of expanders. We will show that this is never true for the sequence of abelian groups $ \textbf{G}_{j} $ defined in Section \ref{abelian_covers}, thus showing Theorem \ref{main4}. Write
$$ \textbf{G}_{j} = \mathbb{Z}/N_{1}^{(j)}\mathbb{Z}\times \mathbb{Z}/N_{2}^{(j)}\mathbb{Z} \times \dots \times \mathbb{Z}/N_{m}^{(j)}\mathbb{Z}. $$
The space $ L^{2}(\textbf{G}_{j}) $ is spanned by the characters $ \chi_{\alpha} $ given by 
$$ \chi_{\alpha}(x) = \exp\left(2\pi i \sum_{\ell =1}^{m} \frac{\alpha_{\ell}}{N_{\ell}^{(j)}}x_{\ell}\right) $$
where $ x = (x_{1}, \dots, x_{m}) $ and $ \alpha =(\alpha_{1}, \dots, \alpha_{m}) $ with $ \alpha_{\ell}\in \{ 0,\dots, N^{(j)}_{\ell} -1\}. $
Note that the trivial character $ \chi_{\alpha} \equiv 1 $ corresponds to $ \alpha = 0 $. Applying the discrete Laplace operator to $ \chi_{\alpha} $ yields
\begin{align*}
\Delta \chi_{\alpha}(x) &=  \chi_{\alpha}(x) - \frac{1}{\vert S_{j}\vert}\sum_{s\in S_{j}} \chi_{\alpha}(x+s)\\
&= \chi_{\alpha}(x) - \frac{1}{\vert S_{j}\vert}\sum_{s\in S_{j}} \exp\left( 2\pi i \sum_{\ell=1}^{m} \frac{\alpha_{\ell}}{N_{\ell}^{(j)}}s_{\ell} \right) \chi_{\alpha}(x)\\
&= \chi_{\alpha}(x) - \frac{1}{\vert S_{j}\vert}\sum_{s\in S_{j}} \cos\left( 2\pi i \sum_{\ell=1}^{m} \frac{\alpha_{\ell}}{N_{\ell}^{(j)}}s_{\ell} \right) \chi_{\alpha}(x)\\
&= \left( 1 - \frac{1}{\vert S_{j}\vert}\sum_{s\in S_{j}} \cos\left( 2\pi i \sum_{\ell=1}^{m} \frac{\alpha_{\ell}}{N_{\ell}^{(j)}}s_{\ell} \right) \right) \chi_{\alpha}(x),
\end{align*}
where we exploited the symmetry of the set $ S_{j} $ in the third line. Thus every character $ \chi_{\alpha} $ is an eigenfunction of $ \Delta $ with eigenvalue 
$$
\lambda_{\alpha}^{(j)} := \frac{1}{\vert S_{j}\vert}\sum_{s\in S_{j}} \left( 1-\cos\left( 2\pi i \sum_{\ell=1}^{m} \frac{\alpha_{\ell}}{N_{\ell}^{(j)}}s_{\ell} \right) \right).
$$
Note that we can view $ S_{j} $ as a subset of $ \{ 0, \dots, N_{1}^{(j)}-1 \}\times \cdots \times \{ 0, \dots, N_{m}^{(j)}-1 \} \subset \mathbb{Z}^{m} $. Since $ S $ is a finite subset of $ \mathrm{PSL}_{2}(\mathbb{R}) $, there exists a constant $ M > 0 $ independent of $ j $ such that $ \max_{s\in S_{j}} \Vert s \Vert_{\infty}\leq M, $ where $ \Vert  s\Vert_{\infty} = \max_{1\leq \ell \leq m} \vert s_{\ell}\vert $ is the supremum norm. Since we assume that $ \vert \textbf{G}_{j}\vert \to +\infty, $ we may assume (after extracting a sequence and reindexing) that $ N_{1}^{(j)} \to +\infty $. Set $ \alpha = (1,0,\dots, 0) $. Then we have
$$ 0 \leq \eta^{(j)}:= \max_{s\in S_{j}}\sum_{\ell=1}^{m} \frac{\alpha_{\ell}}{N_{\ell}^{(j)}}s_{\ell} = \max_{s\in S_{j}} \frac{1}{N_{1}^{(j)}}s_{1} \leq \frac{M}{N_{1}^{(j)}} \to 0 $$
as $ j\to +\infty $. Using $ 1-\cos x \ll x^{2} $ for $ \vert x\vert $ sufficiently small we obtain 
$$ \lambda_{\alpha}^{(j)} \ll (\eta^{(j)})^{2} \to 0 $$
as $ j\to +\infty. $ We need to exclude the possibility that $ \lambda_{\alpha}^{(j)} $ is zero. Note that $ \mathcal{G}_{j} $ is a connected graph because $ S_{j} $ is a generating set for $ \textbf{G}_{j} $. Hence the zero eigenvalue of the discrete Laplacian is simple and therefore
$$ \lambda_{\alpha}^{(j)} = 0 \Leftrightarrow \alpha = 0. $$
In particular, for $ \alpha = (1, 0, \dots, 0) $ we have $ \lambda_{\alpha}^{(j)} >0 $. We have thus shown that the spectral gap $ \lambda_{1}(\mathcal{G}_{j}) $ of $ \mathcal{G}_{j} $ tends to zero as $ j\to +\infty $, up to a sequence extraction. By Proposition \ref{Spectral Gap and Cheeger} this implies that the $ \mathcal{G}_{j} $ do not form a family of expanders. The proof of Theorem \ref{main4} is therefore complete. 
\end{proof}

\subsection{Proof of Proposition \ref{expanders and gap}}

A very similar statement to that of Proposition \ref{expanders and gap} was given by Gamburd \cite[Section~7]{Gamburd1}. The key ingredient in Gamburd's proof is Fell's continuity of induction and we will follow this line of thought. 

For the remainder of this section set $ G = \mathrm{SL}_{2}(\mathbb{R}) $ and let $ \hat{G} $ be its unitary dual, that is, the set of equivalence classes of (continuous) irreducible unitary representations of $ G $. We endow the set $ \hat{G} $ with the \textit{Fell topology}. We refer the reader to \cite{fell} and  \cite[Chapter~F]{Property(T)} for more background on the Fell topology. A representation of $ G $ is called \textit{spherical} if it has a non-zero $ K $-invariant vector, where $ K = \mathrm{SO}(2) $. Let us consider the subset $ \hat{G}^{1} \subset \hat{G} $ of irreducible spherical unitary representations. 

According to Lubotzky \cite[Chapter~5]{Lubotzky}, the set $ \hat{G}^{1} $ can be parametrized as 
$$ \hat{G}^{1} = i\mathbb{R}^{+} \cup \left[ 0, \frac{1}{2}\right], $$
where $ s\in i\mathbb{R}^{+} $ corresponds to the spherical unitary principal series representations, $ s\in (0, \frac{1}{2}) $ corresponds to the complementary series representation, and $ s = \frac{1}{2} $ corresponds to the trivial representation. See also Gelfand, Graev, Pyatetskii-Shapiro \cite[Chapter~1 $ \S 3 $]{Gelfand} for a classification of the irreducible (spherical and non-spherical) unitary representations with a different parametrization. Moreover the Fell topology on $ \hat{G}^{1} $ is the same as that induced by viewing the set of parameters $ s $ as a subset of $ \mathbb{C} $, see \cite[Chapter~5]{Lubotzky}. In particular, the spherical unitary principal series representations are bounded away from the identity.

Let us now recall the connection between the exceptional eigenvalues $ \lambda \in (0, \frac{1}{4}) $ and the complementary series representation.  Consider the (left) quasiregular representation $ (\lambda_{G/\Gamma}, L^{2}(G/\Gamma)) $ of $ G $ defined by 
$$ \lambda_{G/\Gamma}(g)f(h\Gamma) = f(hg^{-1}\Gamma). $$ 
(We will denote this representation simply by $ L^{2}(G/\Gamma) $.) Define the function $ s(\lambda) = \sqrt{1/4 - \lambda} $ for $ \lambda\in (0, \frac{1}{4}) $. Then, $ \lambda\in (0, \frac{1}{4}) $ is an exceptional eigenvalue of $ \Delta_{\Gamma\setminus \mathbb{H}} $ if and only if the complementary series $ \pi_{s(\lambda)} $ occurs as a subrepresentation of $  L^{2}(G/\Gamma) $. This is the so-called Duality Theorem \cite[Chapter~1$ \S 4$]{Gelfand}.

Let us return to the proof of Proposition \ref{expanders and gap}. Let $ \Gamma $ and $ \Gamma_{j} $ be as in Proposition \ref{expanders and gap}. Let $ \Omega(\Gamma) $ denote eigenvalues of the Laplacian $ \Delta_{X} $ on $ X = \Gamma\setminus \mathbb{H} $. Let $ \lambda_{0}(\Gamma) = \delta(1-\delta) = \inf \Omega(\Gamma) $ denote the bottom of the spectrum. Since $ \Gamma_{j} $ is by assumption a finite-index subgroup of $ \Gamma $, we have $ \delta(\Gamma_{j}) = \delta $ and consequently
$$ \lambda_{0}(\Gamma_{j}) = \lambda_{0}(\Gamma) =:\lambda_{0} $$
for all $ j $.
Let $ V_{s_{0}} $ be the invariant subspace corresponding to the representation $ \pi_{s_{0}} $ and let $ L_{0}^{2}(G/\Gamma_{j}) $ be its orthogonal complement in $ L^{2}(G/\Gamma_{j}). $ For each $ j $ we can decompose the quasiregular representation of $ G $ into direct sum of subrepresentations 
$$ L^{2}(G/\Gamma_{j}) =  L_{0}^{2}(G/\Gamma_{j}) \oplus V_{s_{0}}. $$
Recall that $ \lambda_{0} $ is a simple eigenvalue by the result of Patterson \cite{Patterson}. By the Duality Theorem it follows that $ V_{s_{0}} $ is one-dimensional. The following lemma provides us with a link between uniform spectral gap and representation theory.

\begin{lem}\label{Intermediate}
Let $ \mathcal{R} \subset \hat{G}^{1} $ be the following set:
$$ \mathcal{R} = \bigcup_{j} \{ (\pi, \mathcal{H}) : \pi \text{ is spherical irreducible unitary subrep. of } L^{2}_{0}(G/\Gamma_{j}) \} / \sim, $$
where $ \sim $ denotes the equivalence of representations. Then the following are equivalent.
\begin{itemize}
\item[(i)] There exists $ \varepsilon_{0} > 0 $ such that $ \vert s-\delta\vert >\epsilon_{0} $ for all $ j $ and all non-trivial resonances $ s $ of $ X_{j} $.
\item[(ii)] The representation $ \pi_{s_{0}} $ is isolated in the set $ \mathcal{R} \cup \{ \pi_{s_{0}}\} $ with respect to the Fell topology. 
\end{itemize} 
\end{lem}

\begin{proof}
Since the resonances $ s $ of $ X_{j} = \Gamma_{j}\setminus \mathbb{H} $ with $ \Re(s) > \frac{1}{2} $ correspond to the eigenvalues $ \lambda = s(1-s)\in [\lambda_{0}, \frac{1}{4}) $, the uniform spectral gap condition (i) can be stated as follows. There exists $ \epsilon_{1} > 0 $ such that for all $ j $ we have
\begin{equation}\label{eigenvalue gap}
\Omega(\Gamma_{j}) \cap [0, \lambda_{0} + \varepsilon_{1}) = \{ \lambda_{0} \}.
\end{equation} 
Now we can reformulate (\ref{eigenvalue gap}) in representation-theoretic language. Set $ s_{0} = s(\lambda_{0}). $ Then by the Duality Theorem, there exists $ \epsilon > 0 $ such that for all $ j $ and all $ s\in (s_{0}-\varepsilon, \frac{1}{2}] $, the complementary series representation $ \pi_{s} $ does not occur as a subrepresentation of $ L^{2}(G/\Gamma_{j}). $ Since $ V_{s_{0}} $ is one-dimensional (and each representation $ \pi_{s} $ with $ s\neq \frac{1}{2} $ is infinite-dimensional), (i) is equivalent to 
\begin{equation}\label{isolation}
\mathcal{R} \cap \left( s_{0}-\varepsilon, \frac{1}{2}\right] = \{  s_{0}\}. 
\end{equation}
Since the Fell topology on $ \hat{G}^{1} $ is equivalent to the one induced by viewing $ \hat{G}^{1} $ as the subset $ i\mathbb{R}^{+} \cup \left[ 0, \frac{1}{2}\right] $ of the the complex plane, the equivalence of (i) and (ii) is now evident. 
\end{proof}

Let $ 1_{\Gamma_{j}} $ denote the trivial representation of $ \Gamma_{j} $ on $ \mathbb{C} $. Then the induced representation $ Ind_{\Gamma_{j}}^{\Gamma} 1_{\Gamma_{j}} $ is equivalent to the (left) quasiregular representation $ (\lambda_{\textbf{G}_{j}}, L^{2}(\textbf{G}_{j})) $ of $ \Gamma $  defined by 
$$ (\lambda_{\textbf{G}_{j}}(\gamma)F)(h\Gamma_{j} ) = (\gamma.F)(h\Gamma_{j})  = F(h\gamma^{-1} \Gamma_{j}). $$
The action of $ \Gamma $ on $ L^{2}(\textbf{G}_{j}) $ given by $ \gamma.F = \lambda_{\textbf{G}_{j}}(\gamma)F $ is transitive. Hence the only $ \Gamma $-fixed vectors are the constants. Thus we can decompose the representation of $ \Gamma $ on $ L^{2}(\textbf{G}_{j}) $ into a direct of subrepresentations 
$$ L^{2}(\textbf{G}_{j}) = L_{0}^{2}(\textbf{G}_{j}) \oplus \mathbb{C}, $$
where $ L_{0}^{2}(\textbf{G}_{j}) $ is the subspace of functions orthogonal to the constant function, and $ (1_{\Gamma}, \mathbb{C}) $ does not occur as a subrepresentation of $ L_{0}^{2}(\textbf{G}_{j}). $

Consider the following subset of $ \hat{\Gamma} $:
$$ \mathcal{T} = \bigcup_{j\in \mathbb{N}} \{ (\rho, V) : \rho \text{ is irreducible unitary subrepresentation of } L^{2}_{0}(\textbf{G}_{j}) \} / \sim, $$
We claim the following.

\begin{lem}\label{before expanders}
Assume that one of the equivalent statements in Lemma \ref{Intermediate} holds true. Then the trivial representation $ 1_{\Gamma} $ is isolated in $ \mathcal{T} \cup \{ 1_{\Gamma} \} $ with respect to the Fell topology.
\end{lem}

\begin{proof}
Let $ K $ be a closed subgroup of a locally compact group $ H $. Given a unitary representation $ (\pi, V) $ of $ K $, the induced representation $ Ind_{K}^{H} \pi $ of $ H $ is defined as follows. Let $ \mu $ be a quasi-invariant regular Borel measure on $ H/K $ and set
\begin{equation}\label{induced representation}
Ind_{K}^{H} \pi := \{ f : H\to V : f(hk) = \pi(k^{-1})f(h) \text{ for all } k\in K \text{ and } f\in L^{2}_{\mu}(H/K) \}.
\end{equation}
Note that the requirement $ f\in L^{2}_{\mu}(H/K) $ makes sense, since the norm of $ f(g) $ is constant on each left coset of $ H $. The action of $ G $ on $ Ind_{H}^{G} \pi $ is defined by
$$ g.f(x) = f(g^{-1}x) $$ 
for all $ x,g\in G $, $ f\in Ind_{H}^{G} \pi. $ We also note that the equivalence class of the induced representation $ Ind_{K}^{H} \pi $ is independent of the choice of $ \mu $. We refer the reader to \cite[Chapter~E]{Property(T)} for a more thorough discussion on properties of induced representations.

If two representations $ (\pi_{1}, \mathcal{H}_{1}) $ and $ (\pi_{2}, \mathcal{H}_{2}) $ are equivalent, we write $ \mathcal{H}_{1} = \mathcal{H}_{2} $ by abuse of notation. Using induction by stages (see \cite{Folla-95} or \cite{Gaal-73} for a proof) we have
\begin{align*}
V_{s_{0}} \oplus L_{0}^{2}(G/\Gamma_{j}) &= L(G/\Gamma_{j})\\
&= Ind_{\Gamma_{j}}^{G} 1_{\Gamma_{j}}\\
&= Ind_{\Gamma}^{G} Ind_{\Gamma_{j}}^{\Gamma} 1_{\Gamma_{j}}\\
&= Ind_{\Gamma}^{G} L^{2}(\textbf{G}_{j})\\
&= Ind_{\Gamma}^{G} 1_{\Gamma} \oplus Ind_{\Gamma}^{G} L_{0}^{2}(\textbf{G}_{j})\\
&= V_{s_{0}} \oplus L_{0}^{2}(G/\Gamma) \oplus Ind_{\Gamma}^{G} L_{0}^{2}(\textbf{G}_{j}).
\end{align*}
Choose an index $ j $ and an irreducible unitary subrepresentation $ (\tau, V) $ of $ L_{0}^{2}(\textbf{G}_{j}) $. The above calculation implies that $ Ind_{\Gamma}^{G} \tau $ is a unitary subrepresentation of $ L_{0}^{2}(G/\Gamma_{j}) $. Since $ \tau $ is unitary and irreducible, so is $ Ind_{\Gamma}^{G} \tau $. Moreover $ Ind_{\Gamma}^{G} \tau $ is a spherical representation of $ G $, since any non-zero function $ f\in L^{2}(\mathbb{H}/\Gamma) $ and non-zero vector $ v\in V $ gives rise to a non-zero $ K $-invariant function $ F\in Ind_{\Gamma}^{G} \tau $. Indeed, we have $ \mathbb{H} \cong K\setminus G $, so that we my view $ f $ as function $ f : G\to \mathbb{C} $ satisfying $ f(kg\gamma) = f(g) $ for all $ g\in G, k\in K, \gamma\in \Gamma $. Now one easily verifies that $ F = fv : G\to V $ belongs to $ Ind_{\Gamma}^{G} \tau $ and is invariant under $ K $. In other words, $ Ind_{\Gamma}^{G} \tau $ belongs to $ \mathcal{R} $.

Now suppose the lemma is false. Then there exists a sequence $ (\tau_{n})_{n\in \mathbb{N}} \subset \mathcal{T} $ that converges to $ 1_{\Gamma} $ as $ n\to \infty $. On the other hand, $ \pi_{s_{0}} $ is weakly contained in $ Ind_{\Gamma}^{G} 1_{\Gamma} $. By Fell's continuity of induction \cite{fell} we have 
$$ \pi_{s_{0}} \prec Ind_{\Gamma}^{G} 1_{\Gamma} = \lim_{n\to \infty} Ind_{\Gamma}^{G} \tau_{n} \in \overline{\mathcal{R}},   $$
which contradicts Lemma \ref{Intermediate}.
\end{proof}

We can now prove Proposition \ref{expanders and gap}.

\begin{proof}[Proof of Proposition \ref{expanders and gap}]
Let us recall the definition of the Fell topology on $ \hat{\Gamma} $ (for further reading consult \cite[Chapter~F]{Property(T)}). For an irreducible unitary representation $ (\pi, V) $ of $ \Gamma $, for a unit vector $ \xi\in V $, for a finite set $ Q\subset \Gamma $, and for $ \varepsilon > 0 $ let us define the set  $ W(\pi, \xi, Q, \varepsilon) $ that consists of all irreducible unitary representations $ (\pi', V') $ of $ \Gamma $ with the following property. There exists a unit vector $ \xi'\in V' $ such that
$$ \sup_{\gamma\in Q} \vert \langle \pi(\gamma)\xi, \xi\rangle_{V} - \langle \pi'(\gamma)\xi', \xi'\rangle_{V'} \vert < \varepsilon. $$
The Fell topology is generated by the sets $ W(\pi, \xi, Q, \varepsilon) $. By Lemma \ref{before expanders} and the definition of the Fell topology, there exists $ c_{0} = c_{0}(\Gamma, S) > 0 $ only depending on $ \Gamma $ and the generating set $ S $ of $ \Gamma $, but not on $ j $, such that for all $ F\in L_{0}^{2}(\textbf{G}_{j}) $
\begin{equation}\label{Implication}
\sup_{\gamma\in S} \vert \langle \gamma.F - F, F \rangle_{L^{2}(\textbf{G}_{j})}  \vert  \geq c_{0}\Vert F\Vert^{2}.
\end{equation}
By the Cauchy-Schwarz inequality we have 
$$ \sup_{\gamma\in S}\Vert \gamma .F - F \Vert  \geq c_{0}\Vert F\Vert. $$
Fix a non-empty subset $ A $ of $ \textbf{G}_{j} $ with $ \vert A\vert \leq \frac{1}{2}\vert \textbf{G}_{j}\vert $ and define the function
$$ F(x) = 
\begin{cases}
\vert \textbf{G}_{j}\vert - \vert A\vert, & \text{ if } x\in A\\
-\vert A\vert & \text{ if } x\notin A.
\end{cases}
$$
One can verify that $ F\in L_{0}^{2}(\textbf{G}_{j}) $ and $ \Vert F\Vert^{2} = \vert A\vert\vert \textbf{G}_{j} \vert (\vert \textbf{G}_{j}\vert - \vert A\vert). $ On the other hand,
$$ \Vert \gamma .F - F \Vert^{2} = \vert \textbf{G}_{j}\vert^{2} E_{\gamma}(A, \textbf{G}_{j}\setminus A), $$
where
$$ E_{\gamma}(A,B) := \left\vert \{  x\in \textbf{G}_{j} :  x\in A \text{ and } x\gamma \in B \text{ or } x\in B \text{ and } x\gamma \in A \} \right\vert . $$
Therefore there exists $ \gamma\in S $ such that
\begin{align*}
E_{\gamma}(A, \textbf{G}_{j}\setminus A)
= \frac{\Vert \gamma. F - F \Vert^{2}}{\vert \textbf{G}_{j}\vert^{2}}
\geq \frac{c_{0}^{2}\Vert F\Vert^{2}}{\vert \textbf{G}_{j}\vert^{2}} = c_{0}^{2} \left( 1-\frac{\vert A\vert}{\vert \textbf{G}_{j}\vert} \right) \vert A\vert.
\end{align*}
Thus we obtain a lower bound for the size of the boundary of $ A $ in the graph $ \mathcal{G}_{j} = \mathrm{Cay}(\textbf{G}_{j}, S_{j}) $:
$$ \vert \partial A\vert \geq \frac{1}{2} \sup_{\gamma\in S} E_{\gamma}(A, \textbf{G}_{j}\setminus A) \geq  \frac{c_{0}^{2}}{2} \left( 1-\frac{\vert A\vert}{\vert \textbf{G}_{j}\vert} \right) \vert A\vert \geq \frac{c_{0}^{2}}{4} \vert A\vert. $$
Consequently, $ h(\mathcal{G}_{j}) \geq c_{0}^{2}/4 $ for all $ j $ and thus, the graphs $ \mathcal{G}_{j} $ form a family of expanders. The proof of Proposition \ref{expanders and gap} is complete. 
\end{proof}

 \end{document}